\documentclass[12pt]{amsart}
\usepackage[active]{srcltx}
\usepackage{a4wide}
\usepackage{amsthm,amsfonts,amsmath,mathrsfs,amssymb}
\usepackage{dsfont}

\newcommand{\D}{\mathbb{D}}

\newcommand{\RR}{\mathbb{R}}
\newcommand{\C}{\mathbb{C}}
\newcommand{\Z}{\mathbb{Z}}
\newcommand{\T}{\mathbb{T}}

\newcommand{\Cp}{C_{\varphi}}

\newcommand{\Hi}{{\mathscr{H}}^\infty}
\newcommand{\Ht}{{\mathscr{H}}^2}

\newcommand{\Real}{\operatorname{Re}}
\newcommand{\Imag}{\operatorname{Im}}
\newcommand{\Spec}{\operatorname{Spec}}
\newcommand{\vspan}{\operatorname{span}}

\newtheorem{theorem}{Theorem}[section]

\newtheorem{lemma}{Lemma}[section]
\newtheorem{corollary}{Corollary}[section]

\begin{document}

\title[Approximation numbers of composition operators]{Approximation numbers of composition operators \\ on the $H^2$ space of Dirichlet series}

%\subjclass[2000]{32A05, 43A46}

\author[H. Queff\'{e}lec]{Herv\'{e} Queff\'{e}lec}
\address{Universit\'{e} Lille Nord de France, USTL, Laboratoire Paul Painlev\'{e} U.~M.~R. CNRS 8524, F--59 655 Villeneuve d'Ascq
Cedex, France}
%P.~O.~Box 68 (Gustaf H\"{a}llstr\"{o}min Katu 2B), 00014 University
%of Helsinki, Finland} 
\email{herve.queffelec@univ-lille1.fr}

%\thanks{The first author was supported by the Academy of Finland, projects
%no.\ 113826 and 118765. }

\author[K. Seip]{Kristian Seip}
\address{Department of Mathematical Sciences, Norwegian University of Science and Technology,
NO-7491 Trondheim, Norway} \email{seip@math.ntnu.no}
\thanks{The second author is supported by the Research Council of Norway
grant 227768. This paper was written while the authors participated in the research program \emph{Operator Related Function Theory and Time-Frequency Analysis} at the Centre for Advanced Study at the Norwegian Academy of Science and Letters in Oslo during 2012--2013.}
%They would like to thank all the members of the CAS staff for their particularly kind and efficient hospitality, and the excellent working conditions %offered there.}

%\keywords{Dirichlet series, boundary behaviour}
%\date{November 14, 2007.}

% ----------------------------------------------------------------------

\begin{abstract}
By a theorem of Gordon and Hedenmalm, $\varphi$ generates a bounded composition operator on the Hilbert space $\Ht$ of Dirichlet series
$\sum_n b_n n^{-s}$ with square-summable coefficients $b_n$ if and only if  $\varphi(s)=c_0 s+\psi(s)$, where $c_0$ is a nonnegative integer and $\psi$ a Dirichlet series with the following mapping properties: $\psi$ maps the right half-plane into the half-plane $\Real s >1/2$ if $c_0=0$ and is either identically zero or maps the right half-plane into itself if $c_0$ is positive. 
%An observation of this paper is that the Dirichlet series $\psi$ in this theorem converges uniformly in every half-plane $\Real s >\varepsilon$ %for $\varepsilon>0$. 
It is shown that the $n$th approximation numbers of bounded composition operators on $\Ht$ are bounded below by a constant times $r^n$ for some $0<r<1$ when $c_0=0$ and bounded below by a constant times $n^{-A}$ for some $A>0$ when $c_0$ is positive. Both results are best possible. The case when $c_0=0$, $\psi$ is bounded and smooth up to the boundary of the right half-plane, and $\sup \Real \psi=1/2$, is discussed in depth; it includes examples of non-compact operators as well as operators belonging to all Schatten classes $S_p$. For $\varphi(s)=c_1+\sum_{j=1}^d c_{q_j} q_j^{-s}$ with $q_j$ independent integers, it is shown that the $n$th approximation number behaves as $n^{-(d-1)/2}$, possibly up to a factor $(\log n)^{(d-1)/2}$. %In this situation, norms are estimated in $H^2$ of the $d$-disc, via the so-called %Bohr lift. 
Estimates rely mainly on a general Hilbert space method involving finite linear combinations of reproducing kernels. A key role is played by a recently developed interpolation method for $\Ht$ using estimates of solutions of the $\overline{\partial}$ equation. Finally, by a transference principle from $H^2$ of the unit disc, explicit examples of compact composition operators with approximation numbers decaying at essentially any sub-exponential rate can be displayed.
\end{abstract}

%\date{\today}
\maketitle

\section{Introduction and statement of main results}
%%%%%%%%%%%%%%%%%%%%%%%%%%%%%%%%%%%%%%%%%%%%%%%%%%%%%%%%
%%%%%%%%%%%%%%%%%%%%%%%%%%%%%%%%%%%%%%%%%%%%%%%%%%%%%%%%
%%%%%%%%%%%%%%%%%%%%%%%%
By a theorem of Gordon and Hedenmalm \cite{GORHED}, we have a complete characterization of the bounded composition operators on the Hilbert space $\Ht$ which consists of all ordinary Dirichlet series $f(s)=\sum_{n=1}^{\infty}b_n n^{-s}$ such that
\[ \| f\|^2_{\Ht} := \sum_{n=1}^\infty \vert b_n\vert^2 <\infty. \]
Starting from this theorem, the present paper will study the rate of decay of the approximation numbers of compact composition operators on $\Ht$. One of our main findings is that this rate of decay depends crucially on a certain parameter in the Gordon--Hedenmalm theorem. In addition, we will reveal, via the so-called Bohr lift, a precise relation between the complex ``dimension" of the composition operator and the decay of its approximation numbers.  

To make sense of the notion of a composition operator on $\Ht$, we begin by observing that, by the Cauchy--Schwarz inequality, $\Ht$ consists of functions analytic in the half-plane $\sigma:=\Real s>1/2$. This means that $C_{\varphi}f:= f\circ \varphi$ defines an analytic function whenever $\varphi$ maps this half-plane into itself. For $C_{\varphi}$ to map $\Ht$ into $\Ht$, we clearly need to require more. In particular, it turns out that we need to consider other half-planes as well, and for that reason we introduce the notation
\[\C_\theta:=\{s=\sigma+i t:\ \sigma>\theta\},\]
where $\theta$ can be any real number. The Gordon--Hedenmalm theorem reads as follows \cite{GORHED}.

\begin{theorem}[Gordon--Hedenmalm's theorem]\label{julia}The function $\varphi$ determines a bounded composition operator $C_\varphi$ on $\Ht$  if and only if 
$$\varphi(s)=c_{0}s+\sum_{n=1}^\infty c_n n^{-s}=:c_{0}s+\psi(s),$$  where $c_0$ is a nonnegative integer and $\psi$ is a Dirichlet series that converges uniformly in $\C_{\varepsilon}$ for every $\varepsilon>0$ and has the following mapping properties: \begin{itemize}
\item[(a)] If $c_0 =0$, then  $\psi(\C_0)\subset \C_{1/2}$.
\item[(b)] If $c_0\geq 1$, then  either $\psi\equiv 0$ or $\psi(\C_0)\subset \C_0$.
\end{itemize}
\end{theorem}

Here we have made a slight strengthening of the original theorem which only says that $\psi$ converges in some half-plane $\C_{\sigma_0}$ and has an analytic continuation to $\C_0$. In addition, we found it convenient in part (b) to state the mapping properties of $\psi$ rather than of the symbol $\varphi$. Our observation is that either of the mapping properties of $\varphi$ given in the original theorem implies the stronger statement in Theorem~\ref{julia} about the convergence of $\psi$. We will present our short proof of this claim in Section~\ref{Gordon} of the present paper. 

The study of compact composition operators on $\Ht$ was initiated in \cite{BAY, BAYA, FIQUVO}. Bayart succeeded in describing the spectrum of such operators \cite{BAYA}, but otherwise results are rather incomplete compared to what is known in the classical case of $H^2(\D)$  \cite{Shap-livre}. The emphasis in \cite{BAYA, FIQUVO} was on membership in the Hilbert--Schmidt class. Our topic---the rate of decay of the approximation numbers $a_n(C_\varphi)$---is a more delicate issue, and our study of it will reveal new phenomena. 

We recall here that the $n$th approximation number $a_n(T)$ of a bounded operator $T$ on a Hilbert space $H$ is the distance in the operator norm from $T$ to the operators of rank $<n$. The approximation number $a_n(T)$ coincides with the $n$th singular number of $T$ \cite[p. 155]{CA-ST-livre}; % page 155, equation (4.4.12))
the operator $T$ is compact if and only if $a_n(T)\to 0$ when $n\to\infty$, and it belongs, by definition, to the Schatten class $S_p$ for   $0<p<\infty $ if
$$\Vert T\Vert^p:=\hbox{Tr}(\vert T\vert^p)=\sum_{n=1}^\infty a_{n}^p<\infty.$$
We note that $S_2$ coincides with the Hilbert-Schmidt class and that $S_p\subset S_q$ when $p<q$. 
Approximation numbers have the ``ideal property'', expressed by the inequality
\begin{equation}\label{ideal} a_{n}(ATB)\leq  \|A \|a_{n}(T) \| B \|.\end{equation}
Here we may assume in general that a second Hilbert space $H'$ is involved and that $A:H\to H'$, $B:H'\to H$ are bounded operators.

Our first theorem gives general lower bounds for $a_n(C_{\varphi})$. Here and in what follows the notation $f(n)\ll g(n)$ or equivalently $g(n)\gg f(n)$ means that there is a constant $C$ such that $f(n)\le C g(n)$ for all $n$ in question.

\begin{theorem}\label{general} Suppose that $c_0$ is a nonnegative integer and that $\varphi(s)=c_0 s+\sum_{n=1}^\infty c_n n^{-s}$ generates a compact composition operator $C_{\varphi}$ on $\Ht$. 
\begin{itemize}
\item[(a)] If $c_0=0$, then $a_n(C_{\varphi})\gg r^n$ for some $0<r<1$.
\item[(b)] If $c_0=1$, then $a_n(C_{\varphi})\gg n^{-\Real c_1-\varepsilon}$ for every $\varepsilon>0$.
\item[(c)] If $c_0> 1$, then  $a_n(C_{\varphi})\gg n^{-A}$ for some $A>0$. 
 \end{itemize}
\end{theorem}

To see that the rate of decay in (b) and (c) is optimal, it suffices to consider the symbol
$\varphi(s)=s+A$ for some $A>0$. Then $C_{\varphi}$ is a diagonal operator with respect to the orthonormal basis $e_n(s):=n^{-s}$ and
$C_{\varphi}e_n=n^{-A}e_n$, whence $a_{n}(C_\varphi)=n^{-A}$. The fact that the lower bound in (a) can not be improved, requires a more elaborate argument, to be given in Section~\ref{examplec0} below. There we will show that $a_n(C_\varphi)\ll r^n$ whenever $c_0=0$ and the closure of $\psi(\C_0)$ is a compact subset of $\C_{1/2}$. We note that similar estimates and other results regarding approximation numbers in the $H^2(\D)$ setting were obtained in \cite{LIQUEROD}, but the contrast between (a) and (b) + (c) has no parallel in the theory of composition operators on $H^2(\D)$.

We may compare the preceding argument regarding diagonal operators with its parallel for the classical Hardy space $H^2(\D)$. The monomials $z^n$ constitute a similar canonical basis for $H^2(\D)$, and for $\varphi(z)=rz$, $C_{\varphi}$  is a diagonal operator with respect to this basis. We get that $a_n(C_{\varphi})= r^{n-1}$, and this shows that the general lower bound $a_n(C_{\varphi})\gg r^n$ found in \cite{LIQUEROD} is best possible. We see that, in contrast, the approximation numbers of diagonal composition operators with respect to $(e_n)$ decay relatively slowly. 

As suggested by Theorem~\ref{general}, $c_0=0$ represents the most interesting and delicate case because a variety of different rates of decay of $a_n(C_{\varphi})$ may occur. We will give a simple example showing that we may have a bound as in part (c) of Theorem~\ref{general} even when $c_0=0$ and $\varphi(\C_{0})$ is ``large'' but restricted in the sense that $\varphi(\C_{0})\subset \C_{\theta}$ for some $\theta>1/2$. But our main concern will be to reveal the relevance of the complex ``dimension" of the symbol $\varphi$. We will make sense of this by restricting to Dirichlet series of the form
$$\varphi(s)=c_1+\sum_{j=1}^d c_{q_j} q_{j}^{-s},$$ where
$d$ can be a positive integer or $d=\infty$ and the positive integers $q_j\ge 2$ are independent. 
%We also assume that all the coefficients $c_{q_j}$ are nonzero. 
If we set $q=(q_j)$ and use multi-index notation, then this means that any integer $n$ can be written as $n=q^{\alpha}$ for at most one multi-index $\alpha$.  For example, $q_1=2$ and $q_2=6$ are independent. The canonical example of an infinite collection of independent integers is the set of prime numbers.
   
An essential characteristic of a symbol of this kind, is the number
  \[ \kappa(\varphi)=\Real c_1-\sum_{j=1}^d |c_{q_j}|.\]
In view of the Gordon--Hedenmalm theorem  (Theorem \ref{julia}), the independence of the $q_j$, and Kronecker's theorem (see \cite{FIQUVO}), $C_{\varphi}$ is bounded if and only if $\kappa(\varphi)\ge 1/2$.
If $\kappa(\varphi)=1/2$, then $C_{\varphi}$ is compact if and only $d>1$, as was proved independently in \cite{BAYA} and \cite{FIQUVO}. 
Moreover, under the same assumption that $\kappa(\varphi)=1/2$, it was shown in \cite{FIQUVO} that  $C_{\varphi}$ is in the Schatten class $S_4$ when $d=2$ and that $C_\varphi$ belongs to $S_2$ if and only if $d>2$. Thus we have $a_n(C_{\varphi})\ll n^{-1/4}$ for $d=2$ and $a_n(C_{\varphi})\ll n^{-1/2}$ for $d>2$. 

Our next result improves these estimates and gives best possible $d$-dependent upper and lower bounds, up to a factor $(\log n)^{(d-1)/2}$. 

 \begin{theorem} \label{d-depend} Let $\varphi(s)=c_1+\sum_{j=1}^d c_{q_j} q_{j}^{-s}$ be a symbol such that the positive integers $q_j\ge 2$ are independent, $c_{q_j}\neq 0$ for $1\le j\le d$, and $\kappa(\varphi)= 1/2$. 
 \begin{itemize}
\item[(a)] If $1\le d<\infty$, then $ n^{-(d-1)/2}\ll a_n(C_{\varphi}) \ll (n/\log n)^{-(d-1)/2}$.
\item[(b)] If $d=\infty$, then $C_\varphi$ belongs to $\bigcap_{p>0} S_p$.
\end{itemize}
\end{theorem}

In particular, it follows that $C_{\varphi}$ is in $\bigcap_{p>2/(d-1)} S_p$ but not in $S_{2/(d-1)}$ for  $1<d<\infty$.

Part (b) of Theorem~\ref{d-depend} shows that there exists a map $\varphi$ that touches the vertical line $\sigma=1/2$ smoothly, but still the associated composition operator $C_{\varphi}$ belongs to all the Schatten classes $S_p$. Theorem~\ref{d-depend} suggests that the most important issue is not the smoothness of $\varphi$ but rather how ``frequently" the curve $t\mapsto \varphi(it)$ gets close to the vertical line $\sigma=1/2$. This statement will become precise as soon as we have transformed $\varphi$ into a function on the polydisc $\D^d$ via the so-called Bohr lift. 
 
Most of our estimates and, in particular, the entire proof of Theorem~\ref{d-depend} rely on a general method, applicable in the context of Hilbert spaces of analytic functions. These techniques for estimating approximation numbers are not new; they can be found  in the proofs of Proposition 6.3 in \cite{LIQUEROD} (lower bounds) and in Theorem 2.3 and Theorem 3.2 in \cite{LIQURO} (respectively upper and lower bounds). As in the closely related paper \cite{QS}, we have tried in what follows to emphasize the generality of the ideas involved in this method.   

The techniques developed to prove Theorem~\ref{d-depend} will give a few other results as well. First, we will obtain an analogue of MacCluer's compactness condition in terms of vanishing Carleson measures in the special case when $c_0=0$ and the symbol $\varphi$ is bounded. Second, we will establish a transference principle, showing that symbols of composition operators on $H^2(\D)$, via left and right composition with two fixed conformal maps, give rise to composition operators on $\Ht$, such that estimates for the approximation numbers carry over from $H^2(\D)$. Consequently, using results from our recent paper \cite{QS}, we may construct explicit examples of composition operators on $\Ht$ with approximation numbers with essentially any prescribed sub-exponential decay. The most precise result in this direction is obtained in the case of slow decay:

\begin{theorem} \label{slowdecay}
Let $g$ be a function on $\RR^+$ such that $g(x)\searrow 0$ when $x\to \infty$ and $g(x^2)/g(x)$ is bounded below. Then there exists a compact composition operator $\Cp$ on $\Ht$ with
\[ a_n(\Cp)=e^{O(1)} g(n)\]
when $n\to\infty$. 
\end{theorem}

Theorem~\ref{slowdecay} says that we may prescribe any slow rate of decay (a negative power of $\log n$ or slower) and find an admissible symbol $\varphi$  such that the approximation numbers $a_n(C_{\varphi})$ descend accordingly. A particular consequence is that there exist compact composition operators on $\Ht$ belonging to no Schatten class $S_p,\ p<\infty$. The search for composition operators on $H^2(\D)$ with this property began with a question of Sarason, followed by Carroll and Cowen's affirmative answer \cite{CARCOW}.

To close this introduction, we give a brief overview of the subsequent sections of the paper. We have in the next section collected some background material about the space $\Ht$, Carleson measures and interpolating sequences, and approximation numbers. This section also contains Bayart's characterization of the spectrum of $C_{\varphi}$ and a classical lemma of Weyl. In Section~\ref{Gordon}, we use Harnack's inequality to obtain the desired adjustment of  the Gordon--Hedenmalm theorem. Section~\ref{method} presents the general method based on reproducing kernels and MacCluer's compactness condition. To obtain more quantitative applications of the general method, we need precise estimates of solutions to the interpolation problem associated with these reproducing kernels. Section~\ref{interpolate} contains the required result; it is based on a technique from \cite{S} relying on estimates of solutions of the $\overline{\partial}$ equation. Section~\ref{proof1} gives the proof of Theorem~\ref{general}. Part (a) and (b) are shown to be quite easy consequences of Bayart's theorem and Weyl's lemma, while part (c) and the rest of the paper rely on the methods developed in Sections~\ref{method} and \ref{interpolate}. Section~\ref{examplec0} presents some examples for the case $c_0=0$, including the one showing that part (a) of Theorem~\ref{general} is best possible. We then prove Theorem~\ref{d-depend} in Section~\ref{proof2}. A crucial part of this proof consists in analyzing the mapping properties of $\varphi$, viewed as a map from $\D^d$ into $\C_{1/2}$. Finally, we establish the transference principle and consequently Theorem~\ref{slowdecay} in Section~\ref{transfer}.  

 We close the paper in Section~\ref{Conclusion} with some remarks intended to summarize the state of affairs concerning compact composition operators on $\Ht$ and to address some interesting challenges for future investigations. 
\section{Preliminaries}\label{background}
%%%%%%%%%%%%%%%%%%%%%%%%%%%%%%%%%%%%%%%%%%%%%%%%%%%%%%%%
%%%%%%%%%%%%%%%%%%%%%%%%%%%%%%%%%%%%%%%%%%%%%%%%%%%%%%%%
%%%%%%%%%%%%%%%%%%%%%%%%

\subsection{Basic facts about $\Ht$}\label{basics} The space $\Ht$ was introduced and studied in \cite{HLS}. As is readily seen, the reproducing kernel $K_w$ of  $\Ht$ is $K_{w}(s)=\zeta(s+\overline{w})$, where $\zeta$ is the Riemann zeta-function. We have
$$\Vert K_w\Vert^2_{\Ht}=K_{w}(w)=\zeta(2\Real w)$$
for every $w$ in $\C_{1/2}$. This means that  
\begin{equation} \label{pointwise}
|f(\sigma+i t)|\le (\zeta(2\sigma))^{1/2} \|f\|_{\Ht} \le ((\sigma-1/2)^{-1/2}+C)\| f\|_{\Ht}
\end{equation}
for every $f$ in $\Ht$ with $C$ an absolute constant.
The embedding inequality   %\cite[p. 140]{M}, 
\cite{M}, \cite[Theorem 4.11]{HLS}  
\begin{equation}\label{embedding} \int_{\tau}^{\tau+1} |f(1/2+it)|^2 dt \le C \| f\|_{\Ht}^2,\end{equation}
which holds for every $f$ in $\Ht$ and an absolute constant $C$ independent of $\tau$,  gives additional information about the boundary behavior of $f$ in $\C_{1/2}$. It is implicit 
 in this statement that the nontangential boundary limit $f(1/2+i t)$ exists for almost every $t$. Indeed, the embedding inequality shows that\footnote{This relation is made explicit in \eqref{embcons} below.} $f(s)/s$ is in the ordinary Hardy space $H^2(\C_{1/2})$, which is defined as the set of functions
$h$ analytic in $\C_{1/2}$ for which
\begin{equation}\label{htwo} \| h\|_{H^2(\C_{1/2})}^2:=\sup_{\sigma>1/2} \frac{1}{2\pi}\int_{-\infty}^{\infty} |h(\sigma+it)|^2 dt
<\infty.\end{equation} Every $h$ in $H^2(\C_{1/2})$ has a nontangential
boundary limit at almost every point of the vertical line
$\sigma=1/2$, and the corresponding limit function $h\mapsto
h(1/2+it)$ is in $L^2(\RR)$; the $L^2$-norm of this function
coincides with the $H^2$-norm defined by  \eqref{htwo}. 

The space $\Hi$ consist of those functions in $\Ht$ that extend to bounded analytic functions in $\C_0$, and we set
\[ \|f\|_{\Hi}:=\sup_{\sigma+i t:\  \sigma>0}| f(\sigma+i t)|. \]
It was proved in \cite{HLS} that $\Hi$ is the multiplier algebra of $\Ht$. Here we will only need the simple fact that $\|f\|_{\Ht}\le \| f\|_{\Hi}$, which is a direct consequence of a classical theorem of Carlson \cite{Ca}.

%We will use the following alternate way of computing the norm of a function $f$ in $\Ht$. If $f$ has an analytic continuation to $\C_{0}$ and is %bounded in every half-plane $C_{\varepsilon}$ for $\varepsilon>0$, then a classical theorem of Carlson \cite{Ca}   gives that
%\begin{equation}\label{Carlson}
 %\| f\|_{\Ht}^2= \lim_{\sigma \to 0} \lim_{T\to\infty} \frac{1}{T} \int_{0}^{T} |f(\sigma+it)|^2 dt.\end{equation}
%By the Gordon--Hedenmalm theorem, Carlson's theorem applies to compute the $\Ht$-norm of $C_{\varphi}f$ whenever $c_0=0$. We will only %use Carlson's theorem when $f$ is in $\Hi$, which consists of those functions in $\Ht$ that extend to bounded analytic functions in $\C_0$. We %set
%\[ \|f\|_{\Hi}:=\sup_{\sigma+i t:\  \sigma>0}| f(\sigma+i t)| \]
%and observe that Carlson's theorem gives $\|f\|_{\Ht}\le \| f\|_{\Hi}$.
 
We will resort to the so-called Bohr lift in the following special situation. Suppose we are given a finite sequence  $q=(q_1,\ldots,q_d)$, where the $d$ positive integers $q_j$ are  independent. Let $\alpha=(\alpha_1,\ldots,\alpha_d)$ be a multi-index with $\alpha_1,\ldots,\alpha_d$ nonnegative integers; to signify this, we write $\alpha\ge 0$. Consider the subspace of $\Ht$ spanned by the basis vectors $e_n(s)=n^{-s}$ for which $n=q^{\alpha}$ for some multi-index $\alpha$. Then this subspace is mapped isometrically, via $n^{-s}\mapsto z^\alpha$, onto the space $H^2(\D^d)$ which consists of all functions
\[ f(z)=\sum_{\alpha\ge 0} b_{\alpha} z^{\alpha} \]
for which $\| f\|_{H^2(\D^d)}^2:=\sum_{\alpha\ge 0} |b_\alpha|^2$. Now the point is that, for almost every $z$ on the distinguished boundary $\T^d$, the radial limit
$f(z):=\lim_{r\to 1} f(rz)$ exists \cite[p. 46]{Ru} and 
\[ \| f\|_{H^2(\D^d)}^{2}=\int_{\T^d} |f(z)|^2 dm_d(z)<\infty, \]
where $m_d$ denotes normalized Lebesgue measure on $\T^d$. Writing $z=(z^{(1)},\ldots,z^{(d)})$, we see that the Bohr lift of the symbol
$\varphi(s):=c_1+\sum_{j=1}^d c_{q_j} q_j^{-s}$ is the linear function
\begin{equation}\label{bohr} \Phi(z)=c_1+\sum_{j=1}^d c_{q_j} z^{(j)};\end{equation}
the Bohr lift applies to $C_\varphi f$ for every $f$ in $\Ht$ and yields the function $f\circ \Phi$. Note that the reproducing kernel of $H^2(\D^d)$ at the point $w$ in $\D^d$ is the Bohr lift of the function 
\begin{equation}\label{extendK} K^q_w(s):=\prod_{j=1}^d\frac{1}{1-\overline{w^{(j)}} q_j^{-s}}= \sum_{\alpha\geq 0} \overline{w}^{\alpha} (q^{\alpha})^{-s}. \end{equation}

What was just said makes sense also if $d=\infty$ and $w$ is a sequence in $\D^\infty\cap \ell^2$. Then $m_{\infty}$ is the Haar measure on $\T^{\infty}$. Setting $q=(p_1,p_2,\ldots)$, where $p_j$ are the prime numbers listed in ascending order, we see that the Bohr lift maps $\Ht$ isometrically onto $H^2(\D^\infty)$. In particular, via \eqref{extendK}, we may define the reproducing kernel for $\Ht$ at every point $w$ of 
$\D^\infty\cap \ell^2$. We will use the notation
\[ K_{w}^{\infty}(s):=\prod_{j=1}^\infty\frac{1}{1-\overline{w^{(j)}} p_j^{-s}} \]
for this generalized reproducing kernel for $\Ht$.%, as well as the notation
%\[ K_{z}^{\infty}(w)=\prod_{j=1}^\infty(1-\overline{w}^{(j)} z^{(j)})^{-1} \hbox{\  for}\  z,w\in \D^\infty\cap \ell^2. \]
%Another way of computing $\Ht$-norms was found by Bayart \cite{BAYA}. It appears as a Littlewood--Paley type formula for $\Ht$:
%\begin{equation}\label{formulehein}\Vert f\Vert^2_{\Ht}=\vert f(+\infty)\vert^2+4\int_{\T^\infty}\int_{0}^\infty \sigma\vert f'_{\chi}(\sigma)\vert^2d
%\sigma dm(\chi).\end{equation}
%Here $f_\chi$ is the randomized version of $f$ in the sense of \cite{HLS}:
%$$f(s)=\sum_{n=1}^\infty b_{n}n^{-s},\quad f_{\chi}(s)=\sum_{n=1}^\infty b_{n}\chi(n)n^{-s}.$$
%The function $n\mapsto \chi(n)$ is the random completely multiplicative function such that the values $\chi(p)$ for prime numbers $p$ are %independent and uniformly distributed on the circle, and $m$ is the Haar measure on the  compact group $\T^{\infty}$.

\subsection{Generalities about Carleson measures and interpolating sequences}
%%%%%%%%%%%%%%%%%%%%%%%%%%%%%%%%%%%%%%%%%%%%%%%%%%%%%%%%
%%%%%%%%%%%%%%%%%%%%%%%%%%%%%%%%%%%%%%%%%%%%%%%%%%%%%%%%
%%%%%%%%%%%%%%%%%%%%%%%% 

In general, if $H$ is a Hilbert space of functions on some measurable set $\Omega$ in $\C$, we say that a nonnegative Borel measure $\mu$ on $\Omega$ is a Carleson measure for $H$ if there exists a positive constant $C$ such that 
\[ \int_{\Omega} |f(z)|^2 d\mu(z)\le C \| f\|_{H}^2 \]
for every $f$ in $H$. The smallest possible $C$ in this inequality is called the Carleson norm of $\mu$ with respect to $H$. %(thus, $K=\Vert j\Vert^2$, where $j:H\to %L^{2}(\mu)$ is the natural embedding). 
We denote it by 
$\|\mu\|_{{\mathcal C}, H}$ and declare that $\|\mu\|_{{\mathcal C}, H}=\infty$ if $\mu$ fails to be a Carleson measure for $H$. If it is clear from the context which Hilbert space $H$ we are dealing with, we sometimes simplify the notation by writing just $\|\mu\|_{{\mathcal C}}$ instead of $\|\mu\|_{{\mathcal C}, H}$.

If the linear functional of point evaluation is bounded at some point $z$ in $\Omega$, then $H$ has a reproducing kernel $K^H_z$ at that point, meaning that $K^H_z$ is an element in $H$ with the property that $f(z)=\langle f, K^H_z\rangle_H$ for every $f$ in $H$. Let $\Omega_0$ denote the subset of $\Omega$ consisting of those points at which $H$ has a reproducing kernel. We then say 
that a sequence $Z=(z_j)$ of distinct points $z_j$ in $\Omega_0$ is a Carleson sequence for $H$ if the measure
\[ \mu_{Z,H}:=\sum_j \| K^H_{z_j}\|_{H}^{-2} \delta_{z_j}\] 
is a Carleson measure for $H$.  Again, we sometimes simply the notation by writing $\mu_Z$ when there is only one Hilbert space $H$ under consideration. We will need the following dual statement about Carleson sequences in terms of the reproducing kernels.

\begin{lemma}\label{disk}
If $Z=(z_j)$ is a Carleson sequence for $H$, then 
\[ \Big\| \sum_{j} b_j K^{H}_{z_j}\Big\|_{H}^2\le \|\mu_{Z,H}\|_{{\mathcal C}, H} \sum_j |b_j|^2 \| K^{H}_{z_j} \|_{H}^2 \]
for every finite sequence of complex numbers $(b_j)$.
\end{lemma}

\begin{proof} Set $f=\sum_j b_j K^H_{z_j}$ and note that, since $\|f\|_{H}^2=\langle f,\sum b_j K^H_{z_j}\rangle_{H}=\sum \overline{b_j} f(z_j)$, we have
\[ \label{begin}\|f\|_{H}^4=\Big\vert\sum_{j} \overline{b_j} f(z_j)\Big\vert^2 \le \sum_{j} |b_j|^2\| K^{H}_{z_j} \|_{H}^2
\sum_{\ell} |f(z_{\ell})|^2 \| K^{H}_{z_\ell} \|_{H}^{-2}. \]
To finish the proof, it suffices to observe that the latter sum on the right-hand side is bounded by $\|\mu_{Z,H}\|_{{\mathcal C},H} \|f\|_{H}^2$.
\end{proof}

We say that a sequence $Z=(z_j)$ of distinct points $z_j$ in $\Omega_0$ is an interpolating sequence for $H$  if the interpolation problem $f(z_j)=a_j$ has a solution $f$ in 
$H$ whenever  the admissibility condition 
\begin{equation} \label{adm} 
\sum_j |a_j|^2 \| K^{H}_{z_j} \|_{H}^{-2}<\infty\end{equation}
holds.  If $Z$ is an interpolating sequence for $H$, then the open mapping theorem shows that there is a constant $C$ such that we can solve $f(z_j)=a_j$ with the estimate 
\[ \| f\|_{H}\le C \left(\sum_j |a_j|^2 \| K^{H}_{z_j} \|_{H}^{-2}\right)^{1/2} \]
whenever \eqref{adm} holds. The smallest $C$ with this property is denoted by $M_H(Z)$, and we call it the constant of interpolation. 
%It follows that an interpolating sequence for $H$ is in particular a Carleson sequence for $H$.

We have again a dual statement involving reproducing kernels.

\begin{lemma}\label{interdual}
If $Z=(z_j)$ is an interpolating sequence for $H$, then
\[ \Big\| \sum_{j} b_j K^H_{z_j}\Big\|_{H}^2\ge [M_H(Z)]^{-2} \sum_j |b_j|^2 \| K^{H}_{z_j} \|_{H}^2 \]
for every finite sequence of complex numbers $(b_j)$.
\end{lemma}
This reformulation is classical and seems to have been observed first by Boas \cite{Boa}.  %By linearization, write \[S:=\big(\sum_{j}\vert b_j\vert^2\| K^{H}_{z_j} \|_{H}^2\big)^{1/2}=\sum_{j} \lambda_j \overline{b_j}\| K^{H}_{z_j} \|_{H}  \hbox{\ with}\  \sum_{j}\vert \lambda_{j}\vert^2=1.\]
%Now, observe that $\sum_{j} \vert \lambda_j\vert \| (K^{H}_{z_j} \|_{H})^2 (\| K^{H}_{z_j} \|_{H})^{-2}=1$, so that $\lambda_j \| K^{H}_{z_j} \|_{H}=f(z_j)$ for some $f\in H$ such that $\| f\|_{H}\le M_{H}(Z)$. We then obtain that 
%\[ S=\langle f,\sum_{j} b_j K_{z_j}\rangle\le M_{H}(Z) \| \sum_{j} b_j  K^{H}_{z_j}   \|_{H}. \]

\subsection{Carleson measures and interpolating sequences for $H^2$ spaces}
%%%%%%%%%%%%%%%%%%%%%%%%%%%%%%%%%%%%%%%%%%%%%%%%%%%%%%%%
%%%%%%%%%%%%%%%%%%%%%%%%%%%%%%%%%%%%%%%%%%%%%%%%%%%%%%%%

We now restrict our attention to the three spaces $H^2(\D)$, $H^2(\C_{1/2})$, and $\Ht$. We have $\Omega=\overline{\D}$ and 
$\Omega_0=\D$ if $H=H^2(\D)$ and otherwise $\Omega=\overline{\C_{1/2}}$ and $\Omega_0=\C_{1/2}$. We note that 
$\| K^H_{z}\|_H^{-2}=1-|z|^2$ when $H=H^2(\D)$, $\| K^H_{s}\|_H^{-2}=2\Real s-1$ when $H=H^2(\C_{1/2})$, and $\| K^H_{s}\|_H^{-2}=[\zeta(2\Real s)]^{-1}$ when $H=\Ht$. We will now state two classical results about Carleson measures and interpolating sequences, and discuss the simplest connections between the three different settings. 

We begin with Carleson's characterization of Carleson measures for $H^2$ on half-planes and discs \cite{Cac}. To state this result, we introduce the following terminology. A closed square $Q$ in $\overline{\C_{1/2}}$ with one of its sides lying on the vertical line $\sigma=1/2$ is called a Carleson square; the side length of $Q$ is denoted by $\ell(Q)$. Likewise, a set of the form 
\[Q(r_0,t_0):=\{z=r e^{i t} \in \overline{\D}: \ r\ge r_0, \ |t-t_0|\le (1-r_0)\pi\}\]
is   declared to be a Carleson square in $\D$, and we set $\ell(Q(r_0,t_0)):=1-r_0$.

\begin{theorem}[Carleson's theorem]\label{carleson} Let $\mu$ be a nonnegative Borel measure on $\overline{\C_{1/2}}$ or $\overline{\D}$ and let $H$ be respectively either $H^2(\C_{1/2})$ or $H^2(\D)$. 
There exists an absolute constant $C$ such that \[\| \mu \|_{{\mathcal C},H}\le C \sup_{Q} \mu(Q)/\ell(Q),\] where the supremum is taken over all Carleson squares $Q$ in $\overline{\C_{1/2}}$ or $\overline{\D}$.
\end{theorem}

We do not have a complete characterization of Carleson measures for $\Ht$, but there is a simple relation to $H^2(\C_{1/2})$ when $\mu$ is supported on a compact set.

\begin{lemma}\label{firstlem}
If $\mu$ is a Carleson measure for $H^2(\C_{1/2})$, that is supported on the rectangle 
$1/2\le \Real s\le  \theta$, $|\Imag s|\le R$, then
\[ \|\mu\|_{{\mathcal C},  \Ht} \le C(R^2+\theta^2) \| \mu\|_{{\mathcal C}, H^2(\C_{1/2})},\]
where $C$ is an absolute constant.
\end{lemma}

\begin{proof}
Let $f$ be an arbitrary vector in $\Ht$. Setting $F(s)=f(s)/s$, we have, by our assumption on $\mu$,
\[ \int_{\overline{\C_{1/2}}} |f(s)|^2 d\mu(s)\le 
(R^2+\theta^2) \int_{\overline{\C_{1/2}}} |F(s)|^2 d\mu(s)\le (R^2+\theta^2) \| \mu \|_{{\mathcal C}, H^2(\C_{1/2})} \| F \|_{H^2(\C_{1/2})}^2.\]
By the embedding inequality \eqref{embedding}, we have 
\begin{equation}\label{embcons} \int_{-\infty}^\infty\vert F(1/2+it)\vert^2 dt\le \sum_{k=0}^{\infty} \frac{1}{k^2 +1/4} \int_{k\le |t|\le k+1} \vert F(1/2+it)\vert^2dt\le C  \Vert f\Vert_{\Ht}^2.\end{equation}
\end{proof}

We turn next to the description of interpolating sequences for $H^2(\C_{1/2})$ (see \cite[pp. 156--158]{Nik}). The pseudohyperbolic distance between two points $s$ and $w$ in $\C_{1/2}$ is 
\begin{equation}\label{eta} \varrho(s,w):=\left|\frac{s-w}{s+\overline{w}-1}\right| =\left(1-\frac{(2\Real s-1)(2\Real w -1)}{|s+\overline{w}-1|^2}\right)^{1/2};\end{equation}
the separation constant of $S=(s_j)$ is 
\[ \eta(S):=\inf_{j\neq k}\varrho(s_j,s_k),\]
and we say that $S$ is separated if $\eta(S)>0$. We also need the quantity 
\begin{equation}\label{deltadef}  \delta(S):=\inf_{j}\prod_{k: k\neq j} \varrho(s_j,s_k),\end{equation}
which yields a more severe notion of separation.
The following theorem was obtained from Carleson's work \cite{Cai} by Shapiro and Shields \cite{SS}. See also \cite[p. 261]{Nik}.
\begin{theorem}[Shapiro--Shields's theorem]\label{CSS}
A sequence $S$ of distinct points in $\C_{1/2}$ is an interpolating sequence for $H^2(\C_{1/2})$ if and only if
\begin{itemize}
\item[(a)] $S$ is separated;
\item[(b)] $S$ is a Carleson sequence for $H^2(\C_{1/2})$.
\end{itemize}
Moreover,
\[ 1/\delta(S)\le M_{H^2(\C_{1/2})}(S) \le \|\mu_{S}\|_{\mathcal C}^{1/2}/\delta(S).\] 
\end{theorem}  
Here the estimate $M_{H^2(\C_{1/2})}(S) \le \|\mu_S\|_{\mathcal C}^{1/2}/\delta(S)$ is obtained from a duality argument that can be found in \cite[p. 227]{ScSe}.
The bound $1/\delta(S)\le M_{H^2(\C_{1/2})}(S)$ is a consequence of the fact that the product of a normalized reproducing kernel at $s$ and a Blaschke product has the largest possible modulus at $s$ among unit vectors that are divisible by that particular Blaschke product. 

Sometimes it will suffice to use the crude estimate
 \begin{equation} \label{crude} 1/\delta(S)\le \exp\left[2\pi(1+2\log(1/ \eta(S)))\|\mu_S\|_{\mathcal C}\right],\end{equation}
which is a consequence of the elementary inequality
\[ 1/\delta(S)\le \sup_k \exp\Big(1/2+\log(1/\eta(S))\sum_{j}\frac{(2\Real s_j-1)(2\Real s_k-1)}{|s_j+\overline{s_k}-1|^2}\Big) \]
(see (1.10) in \cite[p. 279]{Garnett-livre}).

In Section~\ref{interpolate}, we will prove an analogue of Lemma~\ref{firstlem} for interpolating sequences. The lemma established in Section~\ref{interpolate} is a considerably more difficult result than Lemma~\ref{firstlem}.

Finally, we introduce the function
\[ T(z):=1/2+\frac{1-z}{1+z},\]
which is the M\"{o}bius map of $\D$ onto $\C_{1/2}$.

\begin{lemma}\label{Tinvariance}
For an arbitrary sequence of distinct points $Z$ in $\D$, we have \[M_{H^2(\D)}(Z)=M_{H^2(\C_{1/2})}(T(Z)).\] 
\end{lemma}
\begin{proof} 
Since 
\begin{equation}\label{relation} \Real (T(z)-1/2)=\frac{1-|z|^2}{|1+z|^2} \end{equation}
and the map $f\mapsto \sqrt{2}(1+z)^{-1}f(T(z))$ is a unitary map from $H^2(\C_{1/2})$ onto $H^2(\D)$, we infer that
the interpolation problem $g(z_j)=a_j$ in $H^2(\D)$ can be solved as follows when $\sum \vert a_j\vert^2(1-\vert z_j\vert^2)<\infty$.  Set $b_j=[(1+z_j)/\sqrt{2}]a_j$  and observe that 
\[\sum_{j} \vert b_j\vert^{2} [2\Real(T(z_j))-1]=\sum_{j}\vert a_j\vert^2(1-\vert z_j\vert^2)\] by \eqref{relation} and the definition of $b_j$. We can thus find $f$ in $ H^{2}(C_{1/2})$ such that 
$f(T(z_j))=b_j$. Now set  $g(z):=\sqrt{2}(1+z)^{-1}f(T(z))$. Then we have $g$ in $H^{2}(\D)$, as well as $\Vert g\Vert_{H^{2}(\D)}=\Vert f\Vert_{H^{2}(\C_{1/2})}$ and $g(z_j)=a_j$. 
This shows that  $M_{H^2(\D)}(Z)\le M_{H^2(\C_{1/2})}(T(Z))$. Reversing this argument, we obtain similarly
$ M_{H^2(\C_{1/2})}(T(Z))\le M_{H^2(\D)}(Z)$.
\end{proof}

\subsection{Bernstein numbers}
%%%%%%%%%%%%%%%%%%%%%%%%%%%%%%%%%%%%%%%%%%%%%%%%%%%%%%%%
%%%%%%%%%%%%%%%%%%%%%%%%%%%%%%%%%%%%%%%%%%%%%%%%%%%%%%%%
%%%%%%%%%%%%%%%%%%%%%%%% 

We will make use of the following general characterization of $n$th approximation numbers. %in terms of the ``conorms'' $\gamma$.

\begin{lemma}\label{bernstein} Let $T$ be a bounded operator on a Hilbert space $H$. Then
\begin{equation}\label{deuz} a_n(T)=\sup_{\dim E=n}\Big[ \inf_{x\in E, \|x\|=1}\Vert Tx\Vert\Big].  %=:\sup_{\dim E=n}\Big[\gamma(T_{\vert E})\Big].
\end{equation}
\end{lemma} 
The proof is elementary and can be found in \cite{PIE}.  The number defined by the right-hand side of \eqref{deuz} is called the $n$th Bernstein number of $T$. 

One may use Lemma~\ref{bernstein} to establish lower bounds for $a_n(T)$. The efficiency of this method depends on whether a good choice of $E$ can be made. In our case, when $T=C_{\varphi}^{*}$, we will  take advantage of the relation 
 \begin{equation}\label{identite}C_{\varphi}^{*}(K_a)=K_{\varphi(a)}\end{equation}
which holds for every point $a$ in $\C_{1/2}$. Similarly, when $\varphi$ is as in Theorem~\ref{d-depend} with Bohr lift given by \eqref{bohr}, 
 \begin{equation}\label{identity}C_{\varphi}^{*}(K^q_w)=K_{\Phi(w)}\end{equation}
 for every point $w$ in $\D^d$, as a consequence of the easily verified relation 
 \[\langle n^{-\varphi(s)}, K_{w}^{q}\rangle=n^{-\Phi(w)}=\langle n^{-s}, K_{\Phi(w)}\rangle  \]
for $n=1,2,\ldots$.  In either case,  we will choose $E$ as a linear span of a suitable finite sequence of reproducing kernels or, more generally, of linear combinations of reproducing kernels.
To succeed with this approach, we need precise results about Carleson measures and interpolating sequences.
  
 \subsection{Bayart's theorem on the spectrum of compact composition operators}
 
In \cite[Theorem 4]{BAYA}, Bayart gave the following general description of the spectrum $\Spec(C_\varphi)$ of  compact composition operators on $\Ht$.  
\begin{theorem}[Bayart's theorem]\label{Fred} Let $\varphi(s)=c_0 s + \sum_{n=1}^\infty c_n n^{-s}$
be a  symbol such that $C_{\varphi}$ is a compact composition operator on $\Ht$.  
\begin{itemize}
\item[(a)] If $c_0=0$, then $\Spec{C_{\varphi}}=\{0,1\}\bigcup \{[\varphi'(\alpha)]^k: \ k\ge 1\}$, where $\alpha$ is the fixed point of $\varphi$ in $\C_{1/2}$.
\item[(b)] If $c_0=1$, then $\Spec(C_\varphi)=\{0,1\}\bigcup \{k^{-c_1}:\  k\ge 1\}$.
\item[(c)] If $c_0>1$,  then $\Spec(C_\varphi)=\{0,1\}$.
\end{itemize}
\end{theorem}

When $c_0\le 1$, Bayart's theorem will lead to nontrivial estimates for $a_n(C_\varphi)$  thanks to the following classical lemma of Weyl \cite[p. 157]{CA-ST-livre}.  
  \begin{lemma}[Weyl's lemma] \label{hermann} Let $(a_n)$ be the sequence of approximation numbers of a compact operator $T$ on a Hilbert space $H$, and let $(\lambda_n)$ be the sequence of nonzero eigenvalues of $T$, arranged in descending order. Then
  \begin{equation}\label{weyl} a_1\cdots a_n\geq \vert \lambda_1\cdots \lambda_n\vert,\quad n=1,2,\ldots \end{equation}
  \end{lemma}
Weyl's lemma was used in a similar context in \cite{LIQUEROD}. 

\section{Range and convergence of the Dirichlet series $\psi$ in Theorem~\ref{julia}}\label{Gordon}

We recall the necessary and sufficient condition for boundedness of  $C_\varphi$ given in the original theorem of Gordon and Hedenmalm \cite{GORHED}: We may write $$\varphi(s)=c_{0}s+\sum_{n=1}^\infty c_n n^{-s}=:c_{0}s+\psi(s),$$  where $c_0$ is a nonnegative integer and $\psi$ is an analytic function
in $\C_0$ that can be represented by a convergent Dirichlet series in some half-plane $\C_{\sigma_0}$. Moreover, $\varphi$  has the following mapping properties: \begin{itemize}
\item[(a)] If $c_0 =0$, then  $\varphi(\C_0)\subset \C_{1/2}$.
\item[(b)] If $c_0\geq 1$, then $\varphi(\C_0)\subset \C_0$.
\end{itemize}
We will prove that this condition implies the condition stated in Theorem~\ref{julia}.

To begin with, we note that if the function $\psi$ is nontrivial and the above condition (b) holds, then $\psi$ also maps $\C_0$ to $\C_0$. Indeed, if $\varphi(\C_{0})\subset\C_{0}$, then $c_0s+\psi(s)$ is a Herglotz function in $\C_0$, which in particular means that $\psi$ can be expressed as a Poisson integral of a nonnegative measure along the imaginary axis and so $\psi(\C_0)\subset\C_0$ whenever this measure is nontrivial \cite[p. 17]{Garnett-livre}. We note that this implication concerning the mapping property of $\psi$ was also proved in \cite[Proposition 4.3]{GORHED} by a different argument.

The remaining issue is to show that the condition above implies uniform convergence in every half-plane $\C_{\varepsilon}$, $\varepsilon>0$, of the Dirichlet series representing $\psi$. To this end, we observe first that this Dirichlet series will be uniformly bounded in every half-plane $\C_\theta$ when $\theta>\sigma_0+1$. We fix such an abscissa $\theta$ and choose any number $0<\alpha<1$. Then the function $\psi^\alpha$ is analytic in $\C_0$ and has the property that $|\psi(s)|^{\alpha}\le c \Real [\psi(s)]^{\alpha}$ for a constant $c$ that only depends on $\alpha$. Given any $s=\sigma +it$ in $\C_{\varepsilon},\ \varepsilon>0$, we can now apply Harnack's inequality to the positive harmonic function $u:=\Real \psi^{\alpha}$ at the points $\sigma+it$ and $\theta+i t$. Indeed, the same Herglotz formula as above gives $u(\sigma+it)\leq (\theta/\sigma)u(\theta+it)$ if $0<\sigma\leq \theta$ and $t\in \mathbb{R}$. This implies that $\psi^{\alpha}$ and hence $\psi$ is uniformly bounded in $\C_{\varepsilon}$. By a classical theorem of Bohr \cite{Bo}, it follows that the Dirichlet series representing $\psi$ converges uniformly in every half-plane $\C_{\varepsilon}$. 

We note that this argument establishes a result of independent interest in the general theory of Dirichlet series. We state it as a separate theorem: 

\begin{theorem}
Suppose that $\psi$ is analytic with no zeros in $\C_0$ and that the harmonic conjugate of $\log |\psi|$ is bounded in $\C_0$. If $\psi$ can be represented as a convergent Dirichlet series $\sum_n c_n n^{-s}$ in some half-plane $\C_{\sigma_0}$, then this Dirichlet series converges uniformly in $\C_{\varepsilon}$ for every $\varepsilon>0$.
\end{theorem}  

\begin{proof} The assumption on $\psi$ implies that $\psi^{\alpha}$ maps $\C_0$ into a sector $|\arg s|\le \beta<\pi/2$ if $\alpha>0$ is chosen small enough. This observation allows us to repeat the above argument word for word. \end{proof}
\section{A general method}\label{method}
%%%%%%%%%%%%%%%%%%%%%%%%%%%%%%%%%%%%%%%%%%%%%%%%%%%%%%%%
%%%%%%%%%%%%%%%%%%%%%%%%%%%%%%%%%%%%%%%%%%%%%%%%%%%%%%%%

\subsection{A general theorem for $\varphi$ with $c_0=0$} As we will see in this section, the case $c_0=0$ allows for an interesting interaction with function theory on $\T^{\infty}$. 

We will only consider  symbols $\varphi$ that are bounded analytic functions on $\C_0$. Since then in particular $\varphi$ is in $\Ht$, its Bohr lift $\Phi$ can be viewed as a function in $L^2(\T^\infty)$ with values in $\overline{\C_{1/2}}$, and we may represent it by a boundary function $\Phi^*$ on $\T^\infty$. A key point is the following lemma.  

\begin{lemma}\label{liftpoly}
Suppose that $\varphi(s)=\sum_{n=1}^{\infty} c_n n^{-s}$ and  that $\varphi(\C_0)$ is a bounded subset of $\C_{1/2}$. Then
\[ \| \Cp f \|_{\Ht}^2=\int_{\T^{\infty}}| f(\Phi^{*}(z))|^2 dm_{\infty}(z). \]
\end{lemma}

\begin{proof}
We need to verify that $f(\Phi^{*}(z))$ is the boundary  function of the Bohr lift of $\Cp f$. It is obvious that the boundary function of $\Phi^j$ is $(\Phi^*)^j$ and therefore $n^{-\Phi^*}$ is the boundary function of $n^{-\Phi}$ by a Taylor expansion of $n^{-s}$. Thus the result holds for Dirichlet polynomials. By the Gordon--Hedenmalm theorem, this means that the pullback measure $m_{\infty}\circ( \Phi^{*})^{-1}$ is a Carleson measure for $\Ht$. Since the integral on the right-hand side can be rewritten in terms of this Carleson measure and the set of Dirichlet polynomials is dense in $\Ht$, the result follows. \end{proof}

To obtain more quantitative information from Lemma~\ref{liftpoly}, we introduce for every compact subset $\Omega$ of $\C_{1/2}$ a nonnegative Borel measure $\upsilon_{\varphi,\Omega}$ on $\overline{\C_{1/2}} $ by the requirement that
\begin{equation} \label{defcarl}\upsilon_{\varphi,\Omega}(E):=m_\infty (\{z\in \T^{\infty}: \Phi^*(z)\in E\setminus \Omega \})=m_{\infty}\left((\Phi^*)^{-1}(E \setminus \Omega) \right).\end{equation}
With $\Omega$, we associate the number $\theta:=\inf\{\Real s: s\in \Omega\}>1/2$; with any sequence $s_1,\ldots,s_{n-1}$ of $n-1$ not necessarily distinct points such that $\Real s_j\geq \theta$, we associate the finite Blaschke product 
\begin{equation} \label{blaschke} B(s)=\prod_{j=1}^{n-1}\frac{s-s_j}{s-(1/2+\theta)+\overline{s_j}} \end{equation}
which has modulus $\le 1$ on $\Omega$ and modulus $1$ on the vertical line $\Real s=1/4+ \theta/2<\theta.$ Such a function $B$ will be said to be a Blaschke product adapted to $\Omega$.

Suppose next that $S=(s_j)$ is a sequence of $n$ points in $\C_{1/2}$ such that $\Phi^{-1}(S)\subset \D^\infty\cap \ell^2$. For every finite sequence of distinct points $Z$ in $\Phi^{-1}(S)$, we define:
\[ N_{\Phi}(s_j;Z):=\sum_{z\in Z\cap \Phi^{-1}(s_j)} \|K^{\infty}_{z}\|_{\Ht}^{-2}, \]
which can be thought of as a variant of the Nevanlinna counting function. In the next theorem, we have made a slight abuse of notation by 
viewing (via the Bohr lift) any finite sequence $Z$ of distinct points in $\D^{\infty}$  as a Carleson sequence for $\Ht$, itself viewed as $H^{2}(\T^\infty)$.

\begin{theorem}\label{abovebelow}
Suppose that $\varphi(s)=\sum_{n=1}^{\infty} c_n n^{-s}$ generates a bounded composition operator on $\Ht$.
\begin{itemize}
\item[(a)] Let $\Omega$ be a compact subset of $\C_{1/2}$, suppose that $\theta:=\inf\{\Real s: s\in \Omega\}>1/2$, and  let  $B$ be  an   arbitrary Blaschke product of degree $n-1$ adapted to $\Omega$. If $\varphi(\C_0)$ is a bounded subset of $\C_{1/2}$, then 
\[  a_{n}(\Cp)\le \left(\sup_{s\in \Omega} |B(s)|^2 \zeta(1/2+\theta) +\|\upsilon_{\varphi,\Omega}\|_{{\mathcal C}, \Ht}\right)^{1/2}. \]
\item[(b)]
Let $S$ and $Z$ be finite sets in respectively $\C_{1/2}$ and $\D^\infty\cap \ell^2$ such that $\Phi(Z)=S$ and $S$ has cardinality  $n$.  Then
\[ a_n(\Cp)\ge [M_{\Ht}(S)]^{-1} \|\mu_{Z,\Ht}\|_{{\mathcal C},\Ht}^{-1/2} \inf_{j} \left(N_{\phi}(s_j;Z)\zeta(2 \Real s_j) \right)^{1/2}. \]
\end{itemize}
\end{theorem}

\begin{proof}
We begin with part (a), and  for that purpose use the following rank $n-1$ operator. Let $B$ be an arbitrary finite Blaschke product of degree $n-1$ as defined in \eqref{blaschke}. We let $\Ht_B$ be the subspace of functions $f$ in $\Ht$ that are divisible by $B$. This means that if $s_j$ is a zero of $B$ of order $m$, then any $f$ in $\Ht_B$ has a zero at $s_j$ of order at least $m$.  Let $P_B$ denote the orthogonal projection from $\Ht$ onto $\Ht\ominus \Ht_B$. We set $R_{n-1}:=
C_{\varphi} P_B $ and note that this is an operator of rank at most $n-1$. By the definition of the $n$th approximation number $a_n(C_{\varphi})$, we have 
\begin{equation}\label{fromdef} a_{n}(C_{\varphi})\le \| \Cp - R_{n-1}\|. \end{equation}

  Let $f$ be an arbitrary function in $\Ht$. Then $g :=f-P_B f$ is in $\Ht_B$, and
$B^{-1}g$ has the same supremum as $g$ in  $\C_{1/4+\theta/2}$ by  the maximum modulus principle.  Since this supremum coincides with the supremum on the vertical line $\sigma=1/4+\theta/2$ and $|B(s)|=1$ on this line, we therefore get 
\begin{eqnarray} \sup_{z: \Phi*(z)\in \Omega} |g(\Phi^*(z))|^2 & \le & \sup_{z: \Phi^*(z)\in \Omega} |B(\Phi^*(z))|^2 \sup_{\Real s=1/4+\theta/2} |g(s)|^2   \nonumber \\
& \le &   \sup_{z: \Phi*(z)\in \Omega} |B(\Phi^*(z))|^2\, \zeta(1/2+\theta)
\, \|f\|_{\Ht}^2. \label{pointomega} \end{eqnarray}
In the last step, we used the pointwise estimate \eqref{pointwise} as well as the relation $\|g\|_{\Ht}\le \|f\|_{\Ht}$. We finally apply Lemma~\ref{liftpoly} to compute the norm of $\Cp g =(\Cp -R_{n-1})f$ and note that the desired estimate follows if we use \eqref{pointomega} 
when $ \Phi^*(z)$ is in $\Omega$,  and the definition of $\upsilon_{\varphi,\Omega}$ for other $z$, again taking into account that $\|g\|_{\Ht}\le \|f\|_{\Ht}$ . 

We now turn to part (b). We will apply Lemma~\ref{bernstein} and prepare for this by choosing an appropriate space $E$.
For every $s_j$ in $S$, we set
\[ g_j=[N_{\phi}(s_j;Z)]^{-1} \sum_{z\in Z\cap \Phi^{-1}(s_j)} \|K_{z}^{\infty}\|_{\Ht}^{-2}K_{z}^{\infty} \]
and define the $n$-dimensional space
\[ E=\vspan \{g_1,\ldots,g_n\}. \]
Note that by \eqref{identity} and the definition of $N_{\phi}(s_j;Z)$ we have $\Cp^*g_j=K_{s_j}$. According to Lemma~\ref{bernstein}, we have
\[ a_n(\Cp) \ge   \inf_{f\in E, \|f\|=1}\Vert \Cp^* f\Vert,\]
and it remains therefore to estimate $\Vert\Cp^* f\Vert$ for an arbitrary element $f=\sum_{j=1}^n b_j g_j$ in $ E$ with  $\|f\|_{\Ht}=1$. We start from  
Lemma~\ref{interdual} which gives
\begin{equation}\label{fromid}
\|C_{\varphi}^{\ast}f\|_{\Ht}^{2}= \Big\| \sum_{j} b_j K_{s_j}\Big\|_{\Ht}^2\ge [M_{\Ht}(S)]^{-2} \sum_j |b_j|^2 \| K_{s_j} \|_{\Ht}^2. \end{equation}
To relate the right-hand side of this inequality to the norm of $f$, we observe that
\[f= \sum_{j} \frac{b_j}{N_{\phi}(s_j;Z)}\sum_{z\in Z\cap \Phi^{-1}(s_j)}\|K_{z}^{\infty}\|_{\Ht}^{-2}K_{z}^{\infty}. \]
Using Lemma~\ref{disk}, we therefore get
\[ 1 =\| f\|_{\Ht}^2 \le \| \mu_{Z, \Ht}\|_{{\mathcal C}, \Ht} \sum_{j}  \frac{|b_j|^2}{N_{\phi}(s_j;Z)}.\]
Combining this estimate with \eqref{fromid} and using the fact that $\| K_s\|_{\Ht}^2=\zeta(2 \Real s)$, we arrive 
at part (b).
\end{proof}

We refer to \cite{QS}, where similar estimates were established in the classical setting of $H^2(\D)$. In \cite{QS}, we emphasized the point that in the proof of both inequalities we employed finite-dimensional model subspaces. This is another way of saying that the finite-dimensional spaces 
involved are spanned by reproducing kernels.\footnote{In the case of multiple zeros of the Blaschke product $B$, one should include linear functionals for point evaluation of derivatives up to the prescribed order of each zero in question.}

\subsection{The MacCluer condition} Returning to part (a) of Theorem~\ref{abovebelow}, we look at the following simple choice for $B$. Set 
\[ B(s)=\left(\frac{s-\xi}{s-(1/2+\theta_n)+\xi} \right)^{n-1}\]
for some fixed $\xi$, where $1/2<\theta_n=1/2+\varepsilon_n<\xi$ and $\varepsilon_n\to 0$. We consider the compact set $\Omega=\overline{\varphi(\C_0)}\cap\{s:\ \Real s\geq \theta_n\}$ and note that $B$ is a Blaschke product adapted to $\Omega$.  If $s$ is in $\Omega$, then $2\Real s-\varepsilon_n-1\geq \varepsilon_n$. Hence, using \eqref{eta}, we find that
\begin{eqnarray*} \big\vert B(s)\big\vert^{2} & = & \Big[1 -\frac{(2\Real s-\varepsilon_n -1)(2\xi-\varepsilon_n-1)}{\vert s+\xi -\varepsilon_n -1\vert^2}\Big]^{n-1} \\ & \le & \exp\Big[-(n-1)\frac{(2\Real s-\varepsilon_n -1)(2\xi-\varepsilon_n-1)}{\vert s+\xi -\varepsilon_n -1\vert^2}\Big]\ \le \ \exp(-Cn\varepsilon_n)\end{eqnarray*} for some constant $C$. If we for instance choose $\theta_n=1/2+1/\sqrt{n}$, then
 the first term on the right-hand side of part (a) of Theorem~\ref{abovebelow} will tend to $0$  because $\zeta(1/2+\theta_n)\ll \sqrt n$. We have therefore proved the sufficiency of the following condition for compactness.
\begin{corollary}\label{maccluer}
Suppose that $\varphi(s)=\sum_{n=1}^{\infty} c_n n^{-s}$ and that $\varphi(\C_0)$ is a bounded subset of $\C_{1/2}$. Then $\Cp$ is a compact operator on $\Ht$ if and only if 
\[ \limsup_{\ell(Q)\to 0}  m_{\infty}\circ (\Phi^*)^{-1} (Q)/\ell(Q) = 0.\]
\end{corollary}
According to standard terminology, the corollary says that $\Cp$ is compact if and only if  the pullback measure $m_{\infty}\circ (\Phi^*)^{-1}$ is a vanishing Carleson measure on $\overline{\C_{1/2}}$. This kind of condition for compactness was first found by MacCluer for $H^p$ of the ball in $\C^n$ \cite{MAC}. The argument giving the necessity of the condition in Corollary~\ref{maccluer} uses reproducing kernels in exactly the same standard manner as in \cite{MAC}. We therefore omit this part of the proof. 

\subsection{A general lemma} The presence of the term $c_0 s$ in $\varphi(s)$ is an obstacle for transferring our analysis to $\D^{\infty}$ when $c_0\ge 1$. However, staying in the half-plane $\C_{1/2}$, we obtain the following general scheme which will turn out to be useful. The proof is exactly as the proof of part (b) of Theorem~\ref{abovebelow} and is therefore omitted.

\begin{lemma}\label{belowc1}
Suppose that $\varphi(s)=c_0s+\sum_{n=1}^{\infty} c_n n^{-s}$ determines a bounded composition operator $\Cp$ on $\Ht$. 
Let $S=(s_j)$ and $S'=(s_j')$ be finite sets in $\C_{1/2}$, both of  of cardinality $n$, such that $\varphi(s_j')=s_j$ for every $j$.  Then
\[ a_n(\Cp)\ge [M_{\Ht}(S)]^{-1} \|\mu_{S',\Ht}\|_{{\mathcal C},\Ht}^{-1/2} \inf_{j} \left(\frac{\zeta(2 \Real s_j)}{\zeta(2 \Real s'_j)}\right)^{1/2}. \]
\end{lemma}

\section{Interpolation with estimates from solutions of the $\overline{\partial}$ equation}\label{interpolate}
%%%%%%%%%%%%%%%%%%%%%%%%%%%%%%%%%%%%%%%%%%%%%%%%%%%%%%%%
%%%%%%%%%%%%%%%%%%%%%%%%%%%%%%%%%%%%%%%%%%%%%%%%%%%%%%%%
%%%%%%%%%%%%%%%%%%%%%%%% 
The following is a key lemma that will be used several times throughout this paper. Here we use the notation $S_R$ for the subsequence of points $s_j$ from $S$ that satisfy $|\Imag s_j|<R$.

\begin{lemma}\label{secondlem}
Suppose $S=(s_j=\sigma_j+it_j)$ is an interpolating sequence for $H^2(\C_{1/2})$ and that there exists a number $\theta>1/2$ such that $1/2<\sigma_j \le \theta$ for every $j$. Then there exists a constant $C$,  depending on $\theta$, 
such that
\begin{equation}\label{above} M_{\Ht} (S_R) \le  C [M_{H^2(\C_{1/2})}(S)]^{2\theta+6} R^{2\theta+7/2} \end{equation}
whenever $R\ge \theta+1$. 
\end{lemma}

Thus we need to show that the interpolation problem $F(s_j)=a_j$ for $|t_j|<R$ has a solution $F$ in $\Ht$ satisfying
\begin{equation}\label{required} \| F\|^2_{\Ht} \le C[M_{H^2(\C_{1/2})}(S)]^{4\theta+12}  R^{4\theta+7} \sum_{|t_j|<R} |a_j|^2 (\sigma_j -1/2) \end{equation}
whenever the admissibility condition 
\begin{equation}\label{admR}
\sum_{|t_j|<R} |a_j|^2(\sigma_j-1/2)<\infty
\end{equation} 
is satisfied. It was shown in \cite{OS} that the interpolation problem is solvable, but this result does not give the precise estimate stated in \eqref{required}. To obtain this quantitative result, we will use a technique introduced in \cite{S}. We now give a brief  summary of this method.

By the Paley--Wiener theorem, we may represent $f$ in $H^2(\C_{1/2})$
as\footnote{We allow $\varphi$ and $\Phi$ to have different meanings in this section than elsewhere in this paper.}
\[ f(s)=\int_{0}^\infty \varphi(\xi) e^{-(s-1/2) \xi} d\xi \hbox{\quad with}\quad \varphi\in L^{2}(\mathbb{R}^+) \]
so that by the Plancherel identity $\| f
\|_{H^2(\C_{1/2})}= \|\varphi\|_2$. By an appropriate
discretization of the integral in this representation, we obtained in \cite{S} the following lemma.

\begin{lemma} \label{lem0} Let $N$ be a positive integer. Then for every $\varphi$
in $L^2(\log N,\infty)$, there is a function $F(s)=\sum_{n=N}^\infty
a_n n^{-s}$ in $\Ht$, depending linearly on $\varphi$, such that $\|F\|_{\Ht}\le \|\varphi\|_2$ and
the function
\[ \Phi(s)=\int_{\log N}^\infty \varphi(\xi) e^{-(s-1/2) \xi} d\xi - F(s) \]
enjoys the estimate
\[
|\Phi(s)|\le 2 |s-1/2| N^{-\sigma-1/2} \|\varphi\|_2\] for $s$ in
$\C_{1/2}$.
\end{lemma}

To introduce the second essential ingredient in our solution method, we set
\[ \Omega(R, \tau):=\{s=\sigma+it:\ 1/2 \le \sigma\le \tau, \ -R \le t \le R\}
\]
for positive numbers $R$ and $\tau>1/2$. Lebesgue area measure on $\C$
is denoted by $\omega$. The following simple lemma is again from \cite{S}.
\begin{lemma}
Assume that $R-1\ge \theta>0$, and suppose that $g$ is a continuous function on $\C_{1/2}$ supported on
$\Omega=\Omega(R+2, \theta+2)$ and satisfying $|g(s)|\le \varepsilon$. Then
\[ u(s)=\frac{1}{\pi}\int_{\Omega} \frac{g(z)}{s-z} d\omega (z) \]
solves $\overline{\partial}u=g$ in $\C_{1/2}$ with bounds $ \|
u\|_\infty \le c \varepsilon \log R $ for an absolute constant $c$
(independent of $R$) and \[ |u(s)|\le \frac{R \varepsilon}{\pi \operatorname{dist}(s,\Omega)}.\]
\label{lem1}
\end{lemma}
An important consequence of this lemma is that we have 
 \begin{equation}\label{L2}\Vert u\Vert_2:= \sup_{\sigma>ñ 1/2} \left(\int_{-\infty}^\infty
|u(\sigma+it)|^2 dt\right)^{1/2} \le c' \varepsilon \sqrt{R} \log R \end{equation}
for an absolute constant $c'$.

Let now $B$ be a Blaschke product associated with the sequence $S_R$; this is now a Blaschke product in the half-plane $\C_{1/2}.$ We fix a smooth function $\Theta$ on the closed half-plane $\sigma\ge 1/2$ with the following properties: $\Theta$  is supported on $\Omega(R+2,\theta+2)$ such
that $\Theta(s)=1$ for $s$ in $\Omega(R+1,\theta+1)$ and $|\nabla
\Theta|\le 2$. For a given positive integer $N$, we set
$E_N(s)=N^{-s+1/2}$ and define a linear operator $T_N$ on $E_N
H^2(\C_{1/2})$  as follows. First note that, by a change of variable, the elements $f$ of that space are exactly those of the form   
\begin{equation}\label{T} f(s)=\int_{\log N}^\infty \varphi(\xi) e^{-(s-1/2)\xi} d\xi \end{equation}
with $\| f \|_{H^2(\C_{1/2})}= \|\varphi\|_2.$
Set $\Phi=f-F$, where $F$ is as in Lemma~\ref{lem0}, and let $u$
denote the solution from Lemma~\ref{lem1} to the equation
\[ \overline{\partial} u = \frac{\overline{\partial}(\Theta \Phi)}{B E_N}=\frac{(\overline{\partial}\Theta) \Phi}{B E_N}.\]
Then set 
\[ T_Nf:=\Theta \Phi-BE_N u.\]
It is clear from (\ref{L2}) that $T_Nf$ is in $E_N H^2(\C_{1/2})$ since $\Theta$
has compact support. The virtue of $T_N$ is that $T_N f(s)= \Phi(s)$
for $s$ in $S_R$, i.e., $T_Nf-\Phi$ is divisible by $B$.

We will also need to consider the following extension of $T_N$. Let $\tilde{T}_N$ be the operator from $H^2(\C_{1/2})$ to $E_N
H^2(\C_{1/2})$ defined as follows. Set
\begin{equation}\label{Ttilde} f(s)=\int_{0}^\infty \varphi(\xi) e^{-(s-1/2)\xi} d\xi \end{equation}
and $\Phi=f-F$, where $F$ is as in Lemma~\ref{lem0}. Let again $u$
denote the solution from Lemma~\ref{lem1} to the equation
\[ \overline{\partial} u = \frac{\overline{\partial}(\Theta \Phi)}{B E_N},\]
and set
\[ \tilde{T}_Nf:=\Theta \Phi-BE_N u.\]
It follows again that $\tilde{T}_Nf$ is in $E_N H^2(\C_{1/2})$ and that $\tilde{T}_Nf=\Phi$ on $S_R$.

The following estimates are crucial. 
\begin{lemma}\label{crucial1}
We have the norm estimates \[ \|\tilde{T}_N\|\le C [\delta(S)]^{-2} R^{3/2}(\log R) N^{\theta+3/2}\ \text{and}
\  \|T_N\|\le C [\delta(S)]^{-2} R^{3/2}(\log R) N^{-1},\] where $C$ is a constant depending on 
$\theta$, and $\delta(S)$ is as defined in \eqref{deltadef}.
\end{lemma}

\begin{proof} 
We begin by showing that $|B(s)|$ is bounded below by $c[\delta(S)]^{2}$ for a suitable $c$ when $ \nabla \Theta(s)\neq 0$. To see this, we choose 
$r_j=\min(\sigma_j-1/2,1)\delta(S)/4$ and observe that
\[ | B(s)|\left|\frac{s+\overline{s_j}-1}{s-s_j}\right|\ge \delta(S)-\frac{r_j}{\sigma_j-r_j-1/2}\ge \frac{2}{3}\delta(S)\]
when $|s-s_j|=r_j$ since $B(s)(s+\overline{s_j}-1)/(s-s_j)$ is a Blaschke product with modulus at least $\delta(S)$ at $s_j$ and gradient at most $1/(\Real s-1/2)$. It follows that
\[ |B(s)|\ge \frac{2}{3}\frac{r_j}{(2\sigma_j-1+r_j)} \delta(S)\ge c [\delta(S)]^2 \]
when $|s-s_j|=r_j$. We now set $\Delta_j=\{s: |s-s_j|\le r_j\}$. We have $\Delta_j\subset \Omega(R+1,\theta+1)$ and see that in fact $|B(s)|\ge c [\delta(S)]^2$ on $\C_{1/2}\setminus \bigcup_j \Delta_j$ by the minimum modulus principle. Since $\nabla \Theta (s)\equiv 0$ on every disc $\Delta_j$, the claim follows.

The difference between the two cases is that if $\Phi$ is constructed from \eqref{Ttilde}, then
\[  \left|\frac{\overline{\partial}(\Theta \Phi)}{B E_N}\right|\le C [\delta(S)]^{-2}(R^2+(\theta+1)^2)^{1/2} N^{\theta+3/2} \| \varphi \|_2, \]
while if, on  the other hand, $\Phi$ is constructed from \eqref{T}, then
\[  \left|\frac{\overline{\partial}(\Theta \Phi)}{B E_N}\right|\le C [\delta(S)]^{-2} (R^2+(\theta+1)^2)^{1/2} N^{-1} \| \varphi \|_2;\]
both estimates follow directly from Lemma~\ref{lem0}, the lower bound just established for $|B(s)|$,  and our assumptions on $\Theta$ and $S$. 
Using again the estimates
from Lemma~\ref{lem0} and ~\ref{lem1}, and the assumption that $R\ge \theta+1$, we therefore get respectively
\[ \|\tilde{T}_N f\|_{H^2(\C_{1/2})}^2\le C(R^3+[\delta(S)]^{-4} R^3(\log R)^2 N^{2\theta+3}) \|f\|_{H^2(\C_{1/2})}^2\]
and
\[ \| T_N f\|_{H^2(\C_{1/2})}^2\le C(R^3N^{-2}+[\delta(S)]^{-4}R^3(\log R)^2 N^{-2}) \|f\|_{H^2(\C_{1/2})}^2\]
for a constant $C$ depending on $\theta$. 
%For example, Lemma~\ref{lem0} and (\ref{L2}) give
%\[\| T_N f\|_{H^2(\C_{1/2})}^2\ll\Vert \Theta \Phi\Vert_{2}^{2}+\Vert u\Vert_{2}^{2}\ll \|f\|_{H^2(\C_{1/2})}^2\Big(\int_{-R-2}^{R+2} t^{2}N^{-2}dt
%\Big)+\varepsilon^2 R(\log R)^2,\]
%giving the result.
\end{proof}

\begin{proof}[Proof of Lemma~\ref{secondlem}]
As noted above, it is enough to prove that  $F(s_j)=a_j$ for $|t_j|<R$ has a solution $F$ in $\Ht$ satisfying
\eqref{required} 
whenever the admissibility condition \eqref{admR} is satisfied.

By assumption, the interpolation problem $f(s_j)=a_j$ for $|t_j|<R$ has a solution $f$ in $H^2(\C_{1/2})$ satisfying
\[ \| f\|^2_{H^2(\C_{1/2})} \le \big[M_{H^2(\C_{1/2})}(S)\big]^2  \sum_{|t_j|<R} |a_j|^2 (\sigma_j -1/2) \]
whenever the sum on the right-hand side is finite.
%\[ \sum_{j} |a_j|^2(\sigma_j-1/2)<\infty\] is satisfied. 
By our estimate of  $M_{H^2(\C_{1/2})}(S)$ from below in Theorem~\ref{CSS}, it is therefore enough to show that, given an arbitrary $f$ in $H^2(\C_{1/2})$, we may find a solution $F$ in $\Ht$ to the interpolation problem $F(s_j)=f(s_j)$ such that
\begin{equation}\label{required2} \| F\|_{\Ht}^2 \le C [\delta(S)]^{-4\theta-10} R^{4\theta+7} \| f\|_{H^2(\C_{1/2})}^2\end{equation}
for a positive constant $C$. % where $\alpha$ is the constant from Lemma~\ref{crucial1}.

Let $N$ be a positive integer to be determined later. Set $f_0=f$, $f_1=\tilde{T}_N f_0$, and $f_j=T_{N}f_{j-1}=T_N^{j-1} f_1$ for $j>1$. 
Let $F_j$ be the Dirichlet series in $\Ht$ obtained by applying
Lemma~\ref{lem0} to $f_j$. Then $(F_0+f_1)(s_j)=f_0(s_j)$ for $s_j$ in $S_R$ since
$f_1=f_0-F_0$ on $S_R$, and, more generally, $F_j+f_{j+1}=f_j$ on $S_R$.  Iterating, we get that $F_0+\cdots + F_j+
f_{j+1}$ also coincides with $f_0$ on $S_R$. Since  $\|F_j\|_{\Ht}\leq \Vert f_j\Vert_{H^2(\C_{1/2})}$ , it therefore follows that
\begin{equation}\label{FF} F=\sum_{j=0}^\infty F_j \end{equation}
is in $\Ht$ and $F=f_0$ on $S_R$ if we choose $N$
so large that $\|T_N\|<1$. In view of Lemma~\ref{crucial1}, we see that this is obtained if we set $N= [C [\delta(S)]^{-2} R^{3/2+\varepsilon}]$ for some $\varepsilon>0$ and a sufficiently large constant
$C$. Choosing $\varepsilon$ small enough, we obtain the desired estimate \eqref{required2} from our bound for $\| \tilde{T}_N \|$, cf. Lemma~\ref{crucial1}.

% since $\|F\|_{\Ht}\ll\|f_1\|_{H^2(\C_{1/2})}\ll\|\tilde{T}_N\|\|f\|_{H^2(\C_{1/2})}$. 

\end{proof}

\section{Proof of Theorem~\ref{general}}\label{proof1}
%%%%%%%%%%%%%%%%%%%%%%%%%%%%%%%%%%%%%%%%%%%%%%%%%%%%%%%%
%%%%%%%%%%%%%%%%%%%%%%%%%%%%%%%%%%%%%%%%%%%%%%%%%%%%%%%%
%%%%%%%%%%%%%%%%%%%%%%%% 
\subsection{Proof of part (a) of Theorem~\ref{general}}
%%%%%%%%%%%%%%%%%%%%%%%%%%%%%%%%%%%%%%%%%%%%%%%%%%%%%%%%
%%%%%%%%%%%%%%%%%%%%%%%%%%%%%%%%%%%%%%%%%%%%%%%%%%%%%%%%
%%%%%%%%%%%%%%%%%%%%%%%% 
We will apply Bayart's theorem (Theorem~\ref{Fred}). To this end, we need the following lemma. 
\begin{lemma}\label{trick} Let $\xi$ and $\gamma$ be arbitrary points in $\C_{1/2}$. Then there exists a Dirichlet polynomial $\chi$ such that $\chi( \C_{0})\subset \C_{1/2}$ and 
$$\chi(\gamma)=\xi \hbox{\quad and}\quad \chi'(\gamma)\neq 0.$$
\end{lemma}

\begin{proof} The Dirichlet polynomial $\chi$ will be of the form $\chi(s)=c_1+c_2 2^{-s}$ for suitable $c_1, c_2$.  Choose $\varepsilon>0$ such that 
$$\Real \xi-1/2\geq \varepsilon (1+2^{\Real \gamma});$$ this is possible since by assumption $\xi$ is in  $\C_{1/2}$. Take any $c_1$ such that $\vert \xi-c_1\vert=\varepsilon$ and choose $c_2=(\xi-c_1)2^\gamma$. By construction, we have $\chi(\gamma)=\xi$ and also 
$$\chi'(\gamma)=-(\log 2) c_2 2^{-\gamma}\neq 0.$$
We observe that $\chi$ maps $\C_0$ into $\C_{1/2}$ because 
$$\Real c_1- \frac{1}{2}-\vert c_2\vert\geq \Real \xi-\varepsilon-\frac{1}{2}-\vert \xi-c_1\vert 2^{\Real \gamma}\geq  \Real \xi-\frac{1}{2}-\varepsilon(1+ 2^{\Real \gamma})\geq 0.$$
\end{proof}

We now turn to the  proof of part (a) of Theorem~\ref{general}. Since $\varphi$ is assumed to be non-constant, we may fix a point $\xi$ in $\C_{1/2}$ such that $\varphi'(\xi)\neq 0$. Set $\gamma=\varphi(\xi)$ and $\psi=\chi\circ \varphi$, where $\chi$ is as in  Lemma~\ref{trick}. We have 
\begin{equation}\label{psi}\psi(\xi)=\chi(\gamma)=\xi \quad \text{and} \quad  \psi'(\xi)=\chi'(\gamma)\varphi'(\xi)\neq 0.\end{equation} 
Since $C_\psi =C_\varphi\circ C_\chi$, the ideal property of approximation numbers gives that
\begin{equation}\label{ideal}a_{n}(C_\psi)\leq \Vert C_\chi\Vert a_{n}(C_\varphi).\end{equation}
 This implies that it suffices to establish the lower bound $r^n \ll a_n(C_{\psi})$. We then set $\lambda=\psi'(\xi)$ and $a_j:=a_{j}(C_\psi)$, and note that $0<|\lambda|<1$ since $\xi$ is a fixed point of $\psi$, which is not an automorphism of $\mathbb{C}_{1/2}$. We now invoke Bayart's theorem (Theorem~\ref{Fred}) and Weyl's lemma (Lemma~\ref{hermann}) which together give
\[ a_1\cdots a_n\geq  |\lambda|^{1+\cdots + n}= |\lambda|^{n(n+1)/2}.\] 
Replacing $n$ by $2n$, we obtain
$$|\lambda|^{n(2n+1)}\leq a_1\cdots a_{2n}=\prod_{j=1}^n a_j\times \prod_{j=n+1}^{2n} a_j\leq a_{1}^n a_{n}^{n}.$$
Taking $n$th roots, we get 
$$a_n\geq a_1^{-1}|\lambda|^{2n+1}.$$ 

\subsection{Proof of part (b) of Theorem~\ref{general}.}
%%%%%%%%%%%%%%%%%%%%%%%%%%%%%%%%%%%%%%%%%%%%%%%%%%%%%%%%
%%%%%%%%%%%%%%%%%%%%%%%%%%%%%%%%%%%%%%%%%%%%%%%%%%%%%%%%
%%%%%%%%%%%%%%%%%%%%%%%% 

Let $a_n$ and $\lambda_n$ be respectively the $n$th approximation numbers and the eigenvalues of $C_\varphi$, the latter arranged in descending order. We set $\gamma_1=\Real c_1$ and use Bayart's theorem (Theorem~\ref{Fred}) and Weyl's inequality (Lemma \ref{hermann}): 
$$a_1\cdots a_n\geq \vert \lambda_1\cdots \lambda_n\vert =\prod_{k=1}^n k^{-\gamma_1}\geq n^{-n\gamma_1}. $$ Replacing $n$ by $Nn$, where $N$ is a positive integer larger than 1, we obtain
$$(Nn)^{- Nn\gamma_1}\leq a_1\cdots a_{Nn}=\prod_{j=1}^n a_j\times \prod_{j=n+1}^{Nn} a_j\leq a_{1}^n a_{n}^{(N-1)n}.$$
Taking $(N-1)n$th roots, we get 
$$a_n\geq a_1^{-(N-1)^{-1}} N^{-(1-1/N)^{-1}\gamma_1}n^{- (1-1/N)^{-1}\gamma_1}.$$ 
 This gives the desired result since we may choose
$N$ such that $(1-1/N)^{-1}\gamma_1\leq \gamma_1+\varepsilon$ for any given $\varepsilon>0$.
  
\subsection{Proof of part (c) of Theorem~\ref{general}}
%%%%%%%%%%%%%%%%%%%%%%%%%%%%%%%%%%%%%%%%%%%%%%%%%%%%%%%%
%%%%%%%%%%%%%%%%%%%%%%%%%%%%%%%%%%%%%%%%%%%%%%%%%%%%%%%%
%%%%%%%%%%%%%%%%%%%%%%%% 

Note that Bayart's theorem is of no help when $c_0>1$ since then the spectrum is just $\{0,1\}$.

We will apply Lemma~\ref{belowc1}, and we therefore need to find appropriate sequences $S$ and $S'$. We begin by choosing a number $\sigma_0> 1/2$ and set $s'_k:=\sigma_0+i ak$ and $S'=(s'_k)_{k\in \Z}$. We choose
$a>0$ so large that 
\[ a\ge 3 \sup_{\sigma\ge \sigma_0} \left| \sum_{n=1}^{\infty} c_n n^{-s}\right|=3\sup_{\sigma\geq \sigma_0}\left| \psi(s)\right|, \]
cf. Theorem~\ref{julia}. By Carleson's theorem (Theorem~\ref{carleson}), both $S'$ and $c_0S'$ are Carleson sequences for $H^2(\C_{1/2})$. Since $\vert \varphi(s)-c_0s\vert \leq a/3$ for $s$ in $S'$, the same holds for $S:=\varphi(S')$. 

If $s'_j$ and $s'_k$ are two distinct points in $S'$, then we have
\[ |\varphi(s'_j)-\varphi(s'_k)\vert\geq c_0\vert s'_j-s'_k\vert-\vert \psi(s'_j)\vert-\vert \psi(s'_k)\vert\geq c_{0}a-2a/3\geq a/3.\]
We also have \[\Real \varphi(s')-1/2\leq c_{0}a-1/2+a/3\leq 2c_{0}a\]
for every $s'$ in $S'$. Using that $\varphi(s'_j)+\overline{\varphi(s'_k)}-1=\varphi(s'_j)-\varphi(s'_k)+2\Real \varphi(s'_k)-1$, we therefore get
\[ \eta(S)\ge \inf_{j\neq k} \frac{|\varphi(s'_j)-\varphi(s'_k)|}{4c_0a+|\varphi(s'_j)-\varphi(s'_k)|} \ge \frac{1}{12 c_0+1}.\] 
Since $S$ is both separated and a Carleson sequence for $H^2(\C_{1/2})$, it is an interpolating sequence for $H^2(\C_{1/2})$ in view of Shapiro--Shields's theorem (Theorem~\ref{CSS}). We may now restrict to a finite subsequence $S'_n$ of $S'$ consisting of $n$ points with imaginary parts bounded in modulus by $an/2$ and accordingly set $S_n=\varphi(S'_n)$. Then   
Lemma~\ref{secondlem} gives
\[ M_{\Ht}(S_n)\le C n^c \]
with constants $C$ and $c$ that do not depend on $n$. Since the ratio $\zeta(2 \Real s_j)/\zeta(2 \Real s_j')$ is trivially bounded below, the desired estimate now follows from Lemma~\ref{belowc1} and the Carleson measure estimate for $\mu_{S',\Ht}$ which we obtain from Lemma~\ref{firstlem}. 
 
\section{Compact range and restricted range when $c_0=0$}\label{examplec0}

\subsection{The case of compact range when $c_0=0$}\label{compact}
%%%%%%%%%%%%%%%%%%%%%%%%%%%%%%%%%%%%%%%%%%%%%%%%%%%%%%%%
%%%%%%%%%%%%%%%%%%%%%%%%%%%%%%%%%%%%%%%%%%%%%%%%%%%%%%%%
%%%%%%%%%%%%%%%%%%%%%%%% 
Before embarking on the proof of Theorem~\ref{d-depend}, we will in this section present two examples giving additional information in the case when $c_0=0$. We will use both parts of Theorem~\ref{abovebelow}.

We first consider the case when $c_0=0$ and the closure of $\varphi(\C_0)$ is a compact subset of $\C_{1/2}$. Setting 
$\Omega= \overline{\varphi(\C_0)}$, we obtain $\upsilon_{\varphi,\Omega}=0$, and it suffices to take any point $s_0$ in $\Omega$ and
\[ B(s)=\left(\frac{s-s_0}{s-(1/2+\theta)+\overline{s_0}}\right)^{n-1}. \]
This means that $|B(s)|\le r^{n-1}$, where
\[ r:=\sup_{s\in \Omega} \left|\frac{s-s_0}{s-(1/2+\theta)+\overline{s_0}}\right|<1, \]
and hence $a_n(C_\varphi)\ll r^n$ by part (a) of Theorem~\ref{abovebelow}. This example shows that part (a) of Theorem~\ref{general} is best possible.

\subsection{An example of restricted range and $a_n(C_\varphi)\gg n^{-A}$}
%%%%%%%%%%%%%%%%%%%%%%%%%%%%%%%%%%%%%%%%%%%%%%%%%%%%%%%%
%%%%%%%%%%%%%%%%%%%%%%%%%%%%%%%%%%%%%%%%%%%%%%%%%%%%%%%%
%%%%%%%%%%%%%%%%%%%%%%%% 

We assume again that $c_0=0$. Following \cite{FIQUVO}, we say that $C_\varphi$ has restricted range if $\varphi(\C_0)\subset \C_\theta$ for some $\theta>1/2$. A simple argument of Bayart \cite[Theorem 21]{BAY} shows that $C_{\varphi}$ belongs to the Hilbert--Schmidt class $S_2$ whenever $C_{\varphi}$ has restricted range. We will now give an example showing that, for every $\theta>1/2$,  we may have both  $\varphi(\C_0)\subset \C_\theta$ and $a_n(C_{\varphi})\gg n^{-A}$. This means that Bayart's result is essentially best possible, and we conclude that there is a significant difference between restricted range and compact range.

We consider the map 
\[\varphi(s):=c_1+\frac{1+2^{-s}}{1-2^{-s}},\]
where $\Real c_1 > 1/2$. Then the Bohr lift gives us a function on $\D$, namely
\[ \Phi(z)=c_1+\frac{1+z}{1-z}.\]
We begin by choosing 
\[ z_j=e^{-n^{-2}(1+i (n+j))},\]
with $-n/2<j\le n/2$. We obtain  the explicit estimate
\[ \mu_{Z}(Q)/\ell(Q)\le n^2 (1-e^{-2n^{-2}})/\pi\]
for $Z=(z_j)$ and an arbitrary Carleson square $Q$ in $\D$. Thus Carleson's theorem (Theorem~\ref{carleson}) gives that
$ \|\mu_Z\|_{\mathcal C}$ is bounded by a constant independent of $n$. 
Writing $z_j=r_n e^{i\theta_j}$ with $r_n=e^{-n^{-2}}$ and $\theta_j=-n^{-1}-jn^{-2}$, we get the explicit expression 
  \begin{equation}\label{explicit} s_j:=\Phi(z_j) = c_1+\frac{1+  z_j}{ 1-z_j}=c_1+ \frac{1-r_n^2+i 2r_n \sin \theta_j}{(1-r_n)^2+2r_n (1-\cos \theta_j)}. \end{equation} 
In particular, we find that
\[ \Real s_j\le \Real c_1 + \frac{1-r_n^2}{2r_n (1-\cos \theta_j)}\ll 1 \ \ \text{and} \ \ 
|\Imag s_j| \le |\Imag c_1| +\frac{|\sin \theta_j|}{1-\cos\theta_j} \ll n \]
by our assumption on the arguments $\theta_j$. Hence there are constants $c$ and $\tau$ independent of $n$,  such that the  set $S$ of points $s_j$ is contained in the rectangle $1/2\le \sigma \le \tau$, $|t|\le cn$. In addition, we infer from \eqref{explicit}  that there is a constant $\delta$ independent of $n$ such that
$|\Imag s_j-\Imag s_k|\ge \delta$ whenever $j\neq k$. Therefore, $\mu_{S,H^{2}(C_{1/2})}$ is a Carleson measure, and by Shapiro--Shields's theorem (Theorem~\ref{CSS}), the sequence $(s_j)$ is an interpolating sequence for $H^2(\C_{1/2})$ with constant of interpolation bounded by a constant independent of $n$. Using Lemma~\ref{secondlem}, we see that we have obtained estimates from below for each of the three factors on the right-hand side of part (b) of Theorem~\ref{abovebelow} which enable us to conclude that
$a_n(C_{\varphi})\gg n^{-A}$ for some positive $A$. 
 
Note that when $\Real c_1=1/2$, we have 
\begin{equation}\label{Bayartn} \lim_{s\to 0^+} s\zeta(2\Real \varphi (s))=2/\log 2,\end{equation} which means that $C_{\varphi}$ fails to be compact by Bayart's necessary condition for compactness \cite{BAYA}.  
\section{Proof of Theorem~\ref{d-depend}}\label{proof2} 
 
\subsection{Reduction to the case of real coefficients}
We begin by observing that it is enough to consider the case when $c_1$ is real and the $c_{q_j}$ are negative. To this end, we set
\[ \varphi_{0}(s):=\Real c_1-\sum_{j=1}^d \vert c_{q_j}\vert q_{j}^{-s}. \]
We define two unitary operators $U_1$ and $U_2$ on $\Ht$ in the following way. First, we require $U_1 e_n:=n^{i\Imag c_1}e_n$. Second, we set 
\[\lambda:=(-c_{q_1}/|c_{q_1}|,\ldots,-c_{q_d}/|c_{q_d}|) \]
and require 
\[ U_2 e_n:=\begin{cases} \lambda^{\alpha} e_n, & \text{whenever $n=q^{\alpha}$ for some multi-index $\alpha$}; \\ 
                                       e_n, & \text{otherwise}.
                                       \end{cases} \]
We observe that $C_{\varphi}=U_2 C_{\varphi_0} U_1$ which means that $a_n(C_{\varphi_0})=a_n(C_{\varphi})$ since, by the ideal property \eqref{ideal}, approximation numbers are preserved under unitary transformations.

Henceforth it is  assumed that $c_1>0$ and $c_{q_j}<0$ for $1\le j \le d$.

\subsection{Proof of the estimate from below in Theorem~\ref{d-depend}}

We retain the notation and terminology from Subsection~\ref{basics}. Our plan is to use part (b) of Theorem~\ref{abovebelow}.  

We begin by choosing the sequence $S$. Given a positive integer $n$, we choose  $n$ points 
\[ s_j:=1/2+\nu n^{-2}+i j n^{-2}, \ \ -n/2 < j \le n/2.\] Here $\nu$ is a positive number, independent of $n$, to be determined below. We set $S:=(s_j)$. Recalling that $\mu_S=\sum_{j=1}^n(2\Real s_j-1)\delta_{s_j}$, we obtain from \eqref{eta} that 
\[ \eta(S)\geq\frac{1}{(4\nu^2+1)^{1/2}}\geq (1+2\nu)^{-1} \ \ \text{and} \ \  \sup_Q \mu_S(Q)/\ell(Q) \leq 2\nu, \]
where the supremum is taken over all Carleson squares in $\C_{1/2}$. 
Hence, by Shapiro--Shields's theorem (Theorem~\ref{CSS}) and Lemma~\ref{secondlem}, we have an upper bound for $M_{\Ht}(S)$ that only depends on $\nu$.

We turn to the construction of the sequence $Z$. This is less straightforward than before because $\Phi$ depends on $d$ complex variables, and therefore the pre-image $\Phi^{-1}(s_j)$ is a relatively large subset of  $\D^d$. 

For each $s_j$, we will pick $n^{d-1}$ points $z_\beta(s_j)$ in $\D^d$ with $\beta=(\beta_2,\ldots,\beta_d)$ an index with values in the set $\{1,\ldots,n\}^{d-1}$. We choose
\begin{equation}\label{constr} z^{(\ell)}_\beta:=(1- n^{-2})e^{i\ \beta_{\ell}n^{-2} } 
\end{equation}
for $2\le \ell \le d$ and set  $z^{(\ell)}_\beta(s_j):=z^{(\ell)}_\beta$ for every $j$. Since we want to have $\Phi(z_\beta(s_j))=s_j$, we need to require
\[ z_\beta^{(1)}(s_j)=|c_{q_1}|^{-1}\Big(c_1+\sum_{\ell=2}^d c_{q_\ell} z_\beta^{(\ell)}-s_j\Big),\]
and this must be consistent with the a priori restriction $|z_\beta^{(1)}(s_j)|<1$. Using the assumption on $\Phi$, we infer that 
\[  z_\beta^{(1)}(s_j)= 1-|c_{q_1}|^{-1}\Big(\nu n^{-2}+ijn^{-2}-\sum_{\ell=2}^d |c_{q_\ell}|(1- z_\beta^{(\ell)})\Big), \]
from which it follows that 
\begin{equation}\label{sep}
\Big|z_\beta^{(1)}(s_j)-z_\beta^{(1)}(s_{j'})\Big|=|c_{q_1}|^{-1}n^{-2}|j-j'|.
\end{equation}
Moreover, a computation shows that
\[ |z_\beta^{(1)}(s_j)|^2= \Big(1-|c_{q_1}|^{-1} \big(\nu n^{-2}-\sum_{\ell=2}^d |c_{q_\ell}|(1- \Real z_\beta^{(\ell)})\big)\Big)^2+|c_{q_1}|^{-2}\Big(jn^{-2}+\sum_{\ell=2}^d |c_{q_\ell}|\Imag z_\beta^{(\ell)}\Big)^2.\]
Since $1-\Real z_{\beta}^{(\ell)}$ and the last term on the right-hand side are bounded by a constant times $n^{-2}$, independently of $\beta$, we obtain 
\begin{equation}\label{special} 1-cn^{-2}\le |z_\beta^{(1)}(s_j)|\le 1-n^{-2} \end{equation} when $\nu$ is large enough, where $c$ is a constant independent of $\beta$ and $n$. This requirement determines the value of $\nu$ and ensures that we have both $z_{\beta}(s_j)$ in $\D^d$ and $\Phi(z_\beta(s_j))=s_j$.

It is now clear that $\Vert K_{z_{\beta}(s_j)}^{\infty}\Vert_{\Ht}^{-2}$ is bounded below by a constant times  $n^{-2d}$ and hence that
\[ \inf_j (N(s_j; Z)\zeta(2\Real s_j))^{1/2} \ge c (n^{-2d}\cdot n^{d-1} \cdot n^{2})^{1/2}=cn^{-(d-1)/2} \]
for a positive constant $c$ depending only on $\nu$. In view of part (b) of Theorem~\ref{abovebelow}, it remains only to check that 
$\|\mu_Z\|_{\mathcal C}$ is bounded independently of $n$. This is most easily done via an estimation using reproducing kernels.   

Observe that we may write
\[ K^q_{z_\beta(s_j)}(s)=k_{z_\beta^{(1)}(s_j)}(q_1^{-s})\prod_{\ell=2}^d k_{z_\beta^{(\ell)}} (q_{\ell}^{-s}).\]
This implies that
\[ \sum_{\beta,j} b_{\beta,j} K^q_{z_\beta(s_j)} =\sum_{\beta\in\{1,\ldots,n\}^{d-1}} \prod_{\ell=2}^d k_{z_\beta^{(\ell)}} (q_{\ell}^{-s}) \sum_j b_{\beta,j} k_{z_\beta^{(1)}(s_j)}(q_1^{-s}).\]
If we fix all integers in the index $\beta$ except $\beta_\ell$, $2\le \ell \le d$, then we obtain a sequence $Z_{\ell}:=(z_\beta^{(\ell)})_{\beta_{\ell}=1}^n$ of $n$ points in $\D$. These sequences $Z_{\ell}$ are in fact the same for $2\le \ell \le d$. By explicit computation, we obtain for $\mu_{Z_l}=\sum_{z\in Z_l}(1-\vert z\vert^2)\delta_z$ an absolute upper bound $ \|\mu_{Z_{\ell}}\|_{\mathcal C}\le  2/\pi$.  Iterating Lemma~\ref{disk} $d-1$ times (recall that the $q_j$ are independent), we therefore get
\begin{equation}\label{latter} \|\sum_{\beta,j} b_{\beta,j} K^q_{z_\beta(s_j)}\|_{\Ht}^{2} \le C \sum_{\beta\in\{1,\ldots,n\}^{d-1}}
\prod_{\ell=2}^d (1-|z_\beta^{(\ell)}|^{2})^{-1} \Big\|\sum_j b_{\beta,j} k_{z_\beta^{(1)}(s_j)}\Big\|_{H^2(\D)}^2 \end{equation}
with $C$ a constant independent of $n$. The latter sum involves $n^{d-1}$ sequences $Z_\beta:=(z_\beta^{(1)}(s_j))$ in $\D$.
Thanks to \eqref{sep} and the left inequality in \eqref{special}, the Carleson norms $\|\mu_{Z_\beta}\|_{\mathcal C}$ can be estimated explicitly. In particular, we infer  that  these norms are bounded independently of $\beta$ and $n$. Hence we obtain the  estimate
\begin{eqnarray*}  \|\sum_{\beta,j} b_{\beta,j} K^q_{z_\beta(s_j)}\|_{\Ht}^{2}& \le & C \sum_{\beta,j} |b_{\beta,j}|^{2} (1-|z_{\beta}^{(1)}(s_j)|^{2})^{-1}\prod_{\ell=2}^d (1-|z_\beta^{(\ell)}|^{2})^{-1} \\ & = &
C \sum_{\beta,j} |b_{\beta,j}|^{2} \|K^q_{z_\beta(s_j)}\|_{\Ht}^{2}\end{eqnarray*}
 if we apply Lemma~\ref{disk} to each term in the sum on the right-hand side of \eqref{latter} and use the right inequality in \eqref{special}. This amounts to the desired estimate  $ \|\mu_{Z}\|_{\mathcal C}\le C$.

\subsection{Proof of the estimate from above in Theorem~\ref{d-depend}}\label{belowd}
 %%%%%%%%%%%%%%%%%%%%%%%%%%%%%%%%%%%%%%%%%%%%%%%%%%%%%%%%
%%%%%%%%%%%%%%%%%%%%%%%%%%%%%%%%%%%%%%%%%%%%%%%%%%%%%%%%
%%%%%%%%%%%%%%%%%%%%%%%%
In the proof of the estimate from above, the following elementary lemma will be useful.
\begin{lemma}\label{useful}
Let $C$ be a given positive number. Then there exists another positive number $c$ such that 
\[ \left|\frac{\sigma-\varepsilon+it}{\sigma+\varepsilon+it}\right| \le 1-c\varepsilon \]
whenever $0<\varepsilon\le 1$, $ \varepsilon^2 \le \sigma \le C$, and $t^2\le C\sigma$.
\end{lemma}

\begin{proof}
We write
\[ \left|\frac{\sigma-\varepsilon+it}{\sigma+\varepsilon+it}\right|^2=1-\frac{4\sigma \varepsilon}{\sigma^2+\varepsilon^2+t^2+2\sigma \varepsilon}.\]
Using the assumptions on $t$ and $\sigma$, we get
\[ \sigma^2+\varepsilon^2+t^2+2\sigma \varepsilon\le 2(\sigma^2+\varepsilon^2)+C \sigma \le (3C+2)\sigma, \]
and the result follows with $c=2/(3C+2).$
\end{proof} 

We will now prove the estimate from above in part (a) of  Theorem~\ref{d-depend}. We will first assume that $d$ is finite. The Bohr lift $\Phi$ of $\varphi$ will therefore be a function on $\D^d$. Since $\Phi$ is just a linear function, we may consider $\Phi$ as a function on $\C^d$ and, in particular, $\Phi^*=\Phi$ on $\T^d$. Our plan is to use part (a) of Theorem~\ref{abovebelow}. To this end, we choose $\theta:=1/2 +2\rho^2n^{-2}(\log n)^2$,
\[ \Omega:=\{s=\sigma+it: s\in \overline{\Phi(\D^d)} \ \, \text{and} \ \sigma\ge \theta \}, \]
where $\rho>0$ is a (numerical) parameter to be chosen later, and 
\[ B(s)=\left(\frac{s-1/2-\rho n^{-1}\log n}{s- \theta+\rho n^{-1}\log n}\right)^{n-1}. \]
Observe that the Blaschke product $B$ is adapted to $\Omega$ for $n$ large enough. It may be helpful to keep in mind that the two estimates \eqref{inside} and \eqref{Bzeta} below lead to our restriction on $\theta$; it may be seen from these two relations that the argument would break down if $\theta$ were chosen smaller.
  
To ease the exposition, we write
\[ \Phi(z)=c_1+u+i v.\]
Since $\kappa(\varphi)=1/2$, we have $|u+iv|\le  c_1-1/2=:\delta$ and therefore
\[ v^2\le (\delta-u)(\delta+u)\le C(\delta+u) \]
for a constant $C$ depending on the coefficients $c_{q_j}$. This means that
\[(\Imag (\Phi(z))^2\le C (\Real \Phi(z)-1/2).\]
If we now set $\sigma+it= \Phi(z) -1/4-\theta/2$ and  $\varepsilon=\rho n^{-1}  \log n-\rho^2 n^{-2} (\log n)^2$, then  it  follows from Lemma~\ref{useful} that
\begin{equation}\label{inside}\left|\frac{\Phi(z)-1/2-\rho n^{-1}\log n}{\Phi(z)-\theta+\rho n^{-1}\log n}\right|=\left|\frac{\sigma+it-\varepsilon}{\sigma+it+\varepsilon}\right|\le 
1-c \rho n^{-1}  \log n \end{equation}
for some constant $c>0$ whenever $n$ is sufficiently large and $\Phi(z)$ is in $ \Omega$. (Here we used that $\sigma=\Real \Phi(z)-1/2-  \rho^2 n^{-2} (\log n)^2\geq 1/2[\Real \Phi(z)-1/2]$ and $t=\Imag \Phi(z)$ when $\Phi(z)$ is in $\Omega$.)
We therefore find that
\begin{equation}\label{Bzeta} \sup_{s\in \Omega} |B(s)|^2 \zeta(1/2+\theta)\le 
 C n^{-2c\rho} (n/\log n)^2.
\end{equation}
Thus if we choose $\rho$ so that $2-2c\rho<1-d$, i.e. $\rho >(d+1)/(2c)$, then the first term on the right-hand side of part (a) of Theorem~\ref{abovebelow} can be ignored, and it remains only to estimate $\|\upsilon_{\varphi,\Omega}\|_{\mathcal C}$.  

Since $\upsilon_{\varphi,\Omega}$ is supported on a bounded set, it suffices to show that $\mu=\upsilon_{\varphi,\Omega}$ is a Carleson measure for $H^2(\C_{1/2})$ with 
\begin{equation}\label{Carlesonnorm}
\|\mu\|_{\mathcal{C}} \ll (n/\log n)^{-(d-1)},
\end{equation} 
cf. Carleson's theorem (Theorem~\ref{carleson}). In fact, since $\mu(\C_{1/2+2\rho^2 n^{-2} (\log n)^2})=0$, we only need to estimate $\mu(Q)$ for Carleson squares 
$Q$ of side length at most  $2\rho^2 n^{-2} (\log n)^2$. The following lemma yields the desired estimate \eqref{Carlesonnorm} and finishes thus the proof of part (a) of Theorem~\ref{d-depend}.

\begin{lemma}\label{Carlesonsquare}
There exists a constant $C$, depending only on $\varphi$, such that if $Q$ is a Carleson square in $\C_{1/2}$ of side length $\varepsilon\le 
 2\rho^2 n^{-2} (\log n)^2$, then
 \[ \mu(Q)\le C \varepsilon^{(d+1)/2}. \]
 \end{lemma}

\begin{proof}We assume that $\Phi(z)$ is in the given Carleson square $Q$ of side length $\varepsilon$. We write $z^{(j)}=e^{i \theta_j}$ with $\vert \theta_j\vert\leq \pi$, so that
\[ \varepsilon\ge \Real \Phi(z)-1/2=\sum_{j=1}^d |c_{q_j}|(1-\cos \theta_j)%\gg \sum_{j=1}^d \vert c_{q_j}\vert \theta_{j}^2
.\]
This means that there exists a constant $c$ depending on the coefficients $c_{q_j}$ such that $|\theta_j|\le c \varepsilon^{1/2}$. Therefore, setting
\[ K(\varepsilon):=\big\{z=(e^{i\theta_1},\ldots,e^{i\theta_d}): \ |\theta_j|\le c\varepsilon^{1/2}, \ j=1,\ldots ,d\big\},\]
we have $\Phi^{-1}(Q)\subset K(\varepsilon)$. In addition, since
\begin{equation}\label{imaginary} \Imag \Phi(z)=- \sum_{j=1}^{d} |c_{q_j}| \sin \theta_j, \end{equation}
 there exists another positive constant $c'$ such that 
\[ Q=\big\{ \sigma+it:\ 1/2\le\sigma <1/2+\varepsilon, \ t(Q)-\varepsilon/2< t < t(Q) +\varepsilon/2\big\}\]
for some $t(Q)$ satisfying $|t(Q)| \le c' \varepsilon^{1/2}$. If we now write
\[ R(\varepsilon, Q):=\{ z\in K(\varepsilon): \ t(Q)-\varepsilon/2< \Imag \Phi(z)  < t(Q) +\varepsilon/2\big\}, \]
then clearly $\Phi^{-1}(Q)\subset R(\varepsilon,Q)$, and it suffices therefore to estimate $m_d(R(\varepsilon,Q))$. To this end, we observe, using once more \eqref{imaginary}, that 
we have $z$ in $R(\varepsilon,Q)$ if and only if $z$ is in $K(\varepsilon)$ and 
\[ t(Q)+ \sum_{j=2}^{d} |c_{q_j}| \sin \theta_j-\varepsilon/2 < - |c_{q_1}| \sin \theta_1 < t(Q)+ \sum_{j=2}^{d} |c_{q_j}| \sin \theta_j+\varepsilon/2. \]
From this we infer that for any fixed  value of the $(d-1)$-tuple $(\theta_2,\ldots,\theta_d)$, the first argument $\theta_1$ must be restricted to an interval of length a constant (depending on the coefficients $c_{q_j}$) times $\varepsilon$. By Fubini's theorem, it follows that
\[ m_d(R(\varepsilon,Q))\le C \varepsilon^{(d-1)/2} \times \varepsilon.\]
\end{proof}

Finally, we consider the case when $d=\infty$. The first part of the argument that leads to \eqref{Bzeta}, remains the same. By choosing $\rho$ sufficiently large, we see that this term decays as $n^{-A}$ for arbitrary $A>0$. As to the estimate of $\| \upsilon_{\varphi,\Omega}\|_{\mathcal C}$,  the only difference is that the bound in Lemma~\ref{Carlesonsquare} can be obtained with an arbitrary positive integer $m$ in place of $d$. To get this bound, we replace the set $K(\varepsilon)$ in the proof of Lemma~\ref{Carlesonsquare} by
\[ K_m(\varepsilon):=\big\{z=(e^{i\theta_j})\in \T^{\infty}: \ |\theta_j|\le c\varepsilon^{1/2}, \ j=1,\ldots,m\big\}.\]
 This modification of the preceding proof shows that we have
\[ a_n(C_{\varphi})\ll (n/\log n)^{-(m-1)/2} \]
for every positive integer $m$, whence part (b) of Theorem~\ref{d-depend} follows.

\section{A general transference principle and proof of Theorem~\ref{slowdecay} }\label{transfer}
%%%%%%%%%%%%%%%%%%%%%%%%%%%%%%%%%%%%%%%%%%%%%%%%%%%%%%%%
%%%%%%%%%%%%%%%%%%%%%%%%%%%%%%%%%%%%%%%%%%%%%%%%%%%%%%%%
%%%%%%%%%%%%%%%%%%%%%%%%
We will now study a recipe for transferring  a general composition operator on $H^{2}(\D)$ to a composition operator on $\Ht$. The main point will be to show that decay rates for approximation numbers are preserved or at least not perturbed severely under this transference.  Theorem~\ref{slowdecay} will be shown to be a consequence of this principle. 

In what follows, we will again use a M\"{o}bius transformation
$$T(z):=1/2 +\frac{1-z}{1+z},$$
which maps $\D$ conformally onto $\C_{1/2}$. We will also write 
$$I(s):=2^{-s}$$
and view the function $I$ as a map from $\C_{0}$ onto $\D\backslash\{0\}$. A computation shows that 
$$\| f\circ T\|_{H^2(\D)}^2=\int_{\T} \vert f(1/2+i\tan(t/2)\vert^{2}\, \frac{dt}{2\pi}=\int_{-\infty}^\infty \vert f(1/2+ix)\vert^{2}\frac{dx}{\pi(1+x^2)}.$$ Using the embedding inequality \eqref{embedding} as in the proof of Lemma \ref{firstlem}, we deduce that the composition operator defined by the formula $C_{T}(f):=f\circ T$ is a bounded operator from  $\Ht$ to $H^2(\D)$. Similarly, we can define a (non-surjective) isometry $C_I: H^{2}(\D)\to \Ht$ by setting $C_{I}f(s):=f(I(s))$. If $\omega$ is an analytic self-map of $\D$, we  define an analytic map $\varphi:\C_0\to \C_{1/2}$ by the formula $ \varphi:=T\circ \omega\circ I$, which implies $C_\varphi=C_I\circ C_\omega\circ C_T$.
Note that the Dirichlet series $\varphi$ is then the symbol of a bounded composition operator $C_{\varphi}$ on $\Ht$ with $c_0=0$. By the ideal property of approximation numbers and their preservation under left multiplication by isometries,  we immediately get 
 $$a_{n}(C_\varphi)=a_{n}(C_\omega\circ C_T)\leq \Vert C_T\Vert a_{n}(C_\omega).$$
  We have thus proved the easiest of the two inequalities of the following theorem.

 \begin{theorem}\label{tsf} Let $\omega$ be an analytic self-map of $\D$ such that $\omega(\D)$ has a positive distance to $-1$, and set $\varphi:=T \circ \omega \circ I$. 
 There exist positive constants $c$ and $A$ such that if $Z=(z_j)$ is any finite sequence such that both $Z$ and $\omega(Z)$ consist of $n$ distinct points in $\D$, then
\[  c\, [M_{H^2(\D)}(\omega(Z)) ]^{-A} \|\mu_{Z,H^2(\D)}\|_{{\mathcal C}, H^2(\D)}^{-1/2} \inf_{1\le j\le n} \left(\frac{1-|z_j|^2}{1-|\omega(z_j)|^2}\right)^{1/2}
\le a_n(\Cp) \le \| C_T \| a_n(C_\omega). \]
In particular, $C_\varphi$ is compact as soon as $C_{\omega}$ is compact.
 \end{theorem}
\begin{proof}
The proof is a matter of combining estimates that we have already made. To begin with, we see that the Bohr lift of $\varphi$ is just 
$\Phi:=T\circ \omega$. Thus if we start from an arbitrary sequence $Z=(z_j)$ in $\D$ and set $s_j:=\Phi(z_j)$ and $S:=(s_j)$, then part (b) of Theorem~\ref{abovebelow} gives
\begin{equation}\label{three}  a_n(\Cp)\ge [M_{\Ht}(\Phi(Z))]^{-1} \|\mu_{Z,H^2(\D)}\|_{{\mathcal C},H^2(\D)}^{-1/2} \inf_{1\le j\le n} \left((1-|z_j|^2)\zeta(2 \Real s_j) \right)^{1/2}. \end{equation}
Since $S=T\circ \omega(Z)$ and $\omega(Z)$ is bounded away from $-1$, we have 
\[ \zeta(2 \Real \Phi(z_j)) \ge c (1-|\omega(z_j)|^2)^{-1}. \] 
Finally, we use Lemma~\ref{secondlem} and Lemma~\ref{Tinvariance} to see that  
\[  M_{\Ht}(\Phi(Z))\le C[M_{H^{2}(\C_{1/2})}(\Phi(Z))]^{A}= C[M_{H^2(\D)}(\omega(Z))]^{A}\]
for positive constants $C$ and $A$ depending only on $\omega$. The result follows if we plug the latter two estimates into \eqref{three}. 
\end{proof}
 
The bound from below in Theorem~\ref{tsf} gives a less immediate connection to the approximation numbers of $C_{\omega}$, but we may relate it to the following analogous result for $H^2(\D)$ \cite[Theorem 3.1]{QS}: If $Z=(z_j)$ and $\omega(Z)$ are finite sequences, each consisting of $n$ distinct points in $\D$, then
\begin{equation}\label{h2b} a_n(C_{\omega})\ge [ M_{H^2(\D)}(\omega(Z)) ]^{-1} \|\mu_{Z,H^2(\D)}\|_{{\mathcal C},H^{2}(\D)}^{-1/2} \inf_{1\le j\le n} \left(\frac{1-|z_j|^2}{1-|\omega(z_j)|^2}\right)^{1/2}. \end{equation}
This means that lower bounds obtained from Theorem~\ref{tsf} will decay as a fixed power of the bounds extracted from \eqref{h2b}. We refer to \cite{QS}, where a number of estimates based on \eqref{h2b} are found. 

\begin{proof}[Proof of Theorem~\ref{slowdecay}] In the particular case of approximation numbers of slow decay as in part (a) of Theorem~1.1 of \cite{QS}, we use sequences $Z$ for which $M_{H^2(\D)}(\omega(Z))$ are bounded independently of $Z$. Therefore, we obtain exactly the same lower bound, up to a multiplicative constant, for $a_n(C_\omega)$ and $a_n(C_\varphi)$. Since Corollary 1.1 of \cite{QS} relies on estimates from  part (a) of Theorem 1.1 of \cite{QS}, it may be directly transferred to yield Theorem~\ref{slowdecay} of the present paper. 
\end{proof}
 
\section{Concluding remarks} \label{Conclusion}
 
The result of the preceding paragraph shows that we may associate composition operators on $H^2(\D)$ with a particularly simple subclass of
composition operators on $\Ht$, namely those for which $c_0=0$ and the Bohr lift is a function of only one variable. Thus, by our transference principle, even this subclass yields a remarkably rich collection of composition operators.

However, the most intriguing aspect of the study of composition operators on $\Ht$ appears to be the interplay between function theory in respectively half-planes and polydiscs, exemplified by Theorem~\ref{d-depend}. Our subject has a completely different flavor when the Bohr lift of the symbol is a function of several variables.  Indeed, Theorem~\ref{d-depend} shows that we may find a function $\Phi:\D^\infty\to \C_{1/2}$ such that $\Phi(\D^{\infty})$ touches the vertical line $\sigma=1/2$ ``parabolically" at $1/2$ and still the associated composition operator $C_{\varphi}$ belongs to all Schatten classes $S_p$ for $p>0$. The problem of estimating the decay of the approximation numbers of composition operators associated with more general classes of maps from $\D^d$ into $\C_{1/2}$, is a matter that awaits further study. In particular, it appears as a challenge to understand more generally the role of the complex dimension $d$ in this context. This should be a question about exploring more deeply the ramifications of Theorem~\ref{abovebelow}.

In addition, it would be desirable to have a complete characterization of compactness. One may hope that Corollary~\ref{maccluer} could be generalized to cover also the case $c_0=0$ with unbounded symbols $\varphi$, but an obvious obstacle is that we do not know any general characterization of Carleson measures for $\Ht$.
 
%\section{Acknowledgments} Part of this work was done during a stay of the first named author in the Center for Advanced Study (CAS) of Oslo a%t the occasion of the special year Operator related Function Theory and Time-Frequency  Analysis. He would like to thank all the members of t%he CAS for their particularly kind and efficient hospitality, and the excellent working conditions provided.
\section{Acknowledgement} We would like to thank all the staff members of the Centre for Advanced Study at the Norwegian Academy of Science  and Letters for their particularly kind and efficient hospitality, and the excellent working conditions offered there.


\begin{thebibliography}{BRSHZE}

\bibitem{BAY} F. Bayart, \emph{Hardy spaces of Dirichlet series and
their composition operators}, Monatsh. Math. \textbf{136} (2002),
203--236.

\bibitem{BAYA} F. Bayart, \emph{Compact composition operators on a Hilbert of Dirichlet series}, 
Illinois. J. Math., \textbf{47} (2003), 725--743. 

\bibitem{Boa} R. P. Boas Jr., \emph{A general moment problem}, Amer. J. Math. \textbf{63} (1941), 361--370.

\bibitem{Bo} H. Bohr, \emph{\"{U}ber die gleichm\"{a}ssige Konvergenz Dirichletscher Reihen}, J. Reine Angew. Math. \textbf{143} (1912), 203--211.

%\bibitem{Bo2} H. Bohr, \emph{\"{U}ber die Bedeutung der Potenzreihen unendlich vieler Variabeln in der Theorie der Dirichletschen reihen 
%$\sum a_n/n^s$}, Nachr. Akad. Wiss. Gttingen Math.-Phys. Kl. (1913), 441--488. 

\bibitem{CA-ST-livre} B. Carl and I. Stephani, \emph{ 
Entropy, Compactness and the Approximation of Operators}, Cambridge Tracts in Mathematics \textbf{98}, Cambridge University Press, Cambridge, 1990.

\bibitem{Cai} L. Carleson, \emph{An interpolation problem for bounded analytic functions}, Amer. J. Math. \textbf{80} (1958), 921--930.

\bibitem{Cac} L. Carleson, \emph{Interpolations by bounded analytic functions and the corona problem}, Ann. of Math. (2) \textbf{76} (1962),  547--559.

\bibitem{Ca} F. Carlson, \emph{Contributions \`{a} la th\'{e}orie des s\'{e}ries de Dirichlet}, Ark. Mat. (Note I) \textbf{16} (1922), 1--19. 

\bibitem{CARCOW} T.~Carroll and C.C.~Cowen, \emph{Compact composition operators not in the Schatten classes}, J.Oper. Theory  \textbf{26} (1991), 109--120.

\bibitem{FIQUVO} C.~Finet, H.~Queff\'elec, and A.~Volberg, 
\emph{Compactness of composition operators on a Hilbert  space of Dirichlet series}, 
J. Funct. Anal. \textbf{211} (2004), 271--287. 

\bibitem{Garnett-livre}  J. B. Garnett, \emph{Bounded Analytic Functions}, Graduate Texts in Mathematics \textbf{236}, Springer, New York, 2007.

\bibitem{GORHED} J. Gordon and  H. Hedenmalm,
\emph{The composition operators on the space of Dirichlet series with square-summable coefficients}, 
Michigan Math. J.  \textbf{46}  (1999), 313--329. 

\bibitem{HLS} H. Hedenmalm, P. Lindqvist, and K. Seip, \emph{A Hilbert
space of Dirichlet series and systems of dilated functions in
$L^2(0,1)$}, Duke Math. J. \textbf{86} (1997), 1--37.

%\bibitem{LELIQURO} P. Lef\`evre, D. Li, H. Queff\'elec, and L. Rodriguez-Piazza, 
%\emph{Compact  composition operators on the Dirichlet space and capacity of sets of contact points},
%arXiv:1207.1232, 2012.  
  
%\bibitem{LLQR} P. Lef\`evre, D. Li, H. Queff\'elec, and L. Rodriguez-Piazza, \emph{Approximation numbers of  composition operators on the %Dirichlet space}, arXiv:1104.4451, 2011.

\bibitem{LIQUEROD} D.~Li, H.~Queff\'elec, and L.~Rodriguez-Piazza, 
\emph{On  approximation numbers of  composition operators}, J. Approx.Theory. \textbf{164} (2012), 431--459.

\bibitem{LIQURO} D. Li, H. Queff\'elec, and L. Rodriguez-Piazza,
\emph{Estimates for approximation numbers of some classes of  composition operators on the Hardy space},
Ann. Acad. Sci. Fenn.  \textbf{38}, 2013, 1--18.
 	
%\bibitem{M} H.~L.~Montgomery, \emph{Ten Lectures on the Interface between Analytic Number Theory and
%Harmonic Analysis}, Volume 84 of \emph{CBMS Regional Conference Series in Mathematics}, Amer. Math. Soc., 1994.

\bibitem{MAC} B.~D.~MacCluer,  \emph{Compact composition operators on $H^{p}(B_N)$}, Michigan Math.~J.  \textbf{32} (1985), 237--248.

\bibitem{M} H.~L.~Montgomery and R.~C.~Vaughan, \emph{Hilbert's inequality}, J. London Math. Soc. \textbf{8} (1974), 73--82.

%\bibitem{NOS}, A. Nicolau, J. Ortega-Cerd\`{a}, K. Seip, \emph{The constant of interpolation}, Pacific J. Math. \textbf{213} (2004), 389-Ð398.

\bibitem{Nik} N.Nikolskii, \emph{Treatise on the shift operator}, Springer-Verlag, Berlin, Heidelberg, New-York,  1986.


\bibitem{OS} J.-F.~Olsen and K. Seip, \emph{Local interpolation in Hilbert
spaces of Dirichlet series}, Proc. Amer. Math. Soc. \textbf{136}
(2008), 203--212.


\bibitem{PIE} A. Pietsch, \emph{$s$-numbers of operators in Banach spaces}, Studia Math.
 \textbf{51} (1974), 201--223. 

\bibitem{QS} H. Queff\'{e}lec and K. Seip, \emph{Decay rates for approximation numbers of composition operators}, J. Anal. Math., to appear; available at arXiv:1302.4116v2.

% \bibitem{S1} K. Seip, \emph{Interpolation by Dirichlet series in $\mathcal{H}^\infty$}, Linear and complex analysis, pp. 153--164, 
 %Amer. Math. Soc. Transl. Ser. 2 \textbf{226}, Amer. Math. Soc., Providence, RI, 2009. 

%\bibitem{Ro}J.-P. Rosay, \emph{On the radial maximal function and the Hardy Littlewood maximal function in wedges}, The Madison %Symposium on Complex Analysis (Madison, WI, 1991), pp. 383--398,
%Contemp. Math. \textbf{137}, Amer. Math. Soc., Providence, RI, 1992. 

\bibitem{Ru} W. Rudin, \emph{Function Theory in Polydiscs}, W. A. Benjamin, Inc., New York-Amsterdam, 1969.

\bibitem{ScSe} A.~P.~Schuster and K. Seip, \emph{A Carleson-type condition for interpolation in Bergman spaces}, 
J. Reine Angew. Math. \textbf{497} (1998), 223--233.

\bibitem{S} K. Seip, \emph{Zeros of functions in Hilbert spaces of Dirichlet series}, 
Math. Z. \textbf{274} (2013), 1327--1339.

\bibitem{SS} H.~S.~Shapiro and A.~L.~Shields, \emph{On some interpolation problems for analytic functions}, 
Amer. J. Math. \textbf{83} (1961), 513--532.
 
\bibitem{Shap-livre} J. H. Shapiro, 
\emph{Composition Operators and Classical Function Theory}, 
Universitext: Tracts in Mathematics, Springer-Verlag, New York,  1993.


%\bibitem{St} E.~M. Stein, 

%\bibitem{Zhu-livre} K. Zhu, \emph{Operator Theory in Function Spaces}, Second edition, Mathematical Surveys and Monographs \textbf{138}, %American Mathematical Society, Providence, RI, 2007.


%\bibitem{O} J.-F. Olsen, \emph{Local properties of Hilbert spaces of
%Dirichlet spaces}, J. Funct. Anal. \textbf{261} (2011), 2669--2696.

%\bibitem{OS} J.-F. Olsen and E. Saksman,  \emph{On the boundary
%behaviour of the Hardy spaces of Dirichlet series and a frame bound
%estimate}, J. Reine Angew. Math. \textbf{663} (2012), 33--66.





\end{thebibliography}
\end{document}